\documentclass[11pt,leqno]{amsart}

\usepackage{epstopdf,epsfig}
\usepackage[colorlinks=false]{hyperref}
\usepackage{csquotes}

\usepackage{todonotes} %write orange comments with \todo{this is my comment} Disable comments in the PDF by replacing \usepackage{todonotes} with \usepackage[disable]{todonotes}
\usepackage{xcolor}

\usepackage{amsmath,amsthm,amsfonts,amssymb,latexsym,amscd,enumerate,mathtools}

\usepackage{xy}
\xyoption{all}

\theoremstyle{remark}

\newtheorem{case}{Case} %For case distinctions in proof enviroments

\theoremstyle{plain}

\newcounter{theoremintro} 
\newtheorem{introtheorem}[theoremintro]{Theorem}

\newtheorem*{definition*}{Definition} 
\newtheorem*{theorem*}{Theorem} 
\newtheorem*{lemma*}{Lemma}
\newtheorem*{corollary*}{Corollary}

\newtheorem{theorem}[subsection]{Theorem} 
\newtheorem{lemma}[subsection]{Lemma}
\newtheorem{corollary}[subsection]{Corollary}
\newtheorem{proposition}[subsection]{Proposition}

\newtheorem{question}[subsection]{Question}

\theoremstyle{definition}
\newtheorem{definition}[subsection]{Definition}

\theoremstyle{remark}
\newtheorem{example}[subsection]{Example}
\newtheorem{remark}[subsection]{Remark}

\numberwithin{equation}{section}

\newcommand{\SL}{\mathrm{SL}}
\newcommand{\GL}{\mathrm{GL}}

\newcommand{\cO}{\mathcal{O}}

\newcommand{\FIN}{\mathrm{FIN}}
\newcommand{\FINx}{\mathrm{FIN}^\times}

\newcommand{\Stab}{\mathrm{Stab}}

\newcommand{\F}{\mathbb{F}}
\newcommand{\N}{\mathbb{N}}
\newcommand{\Z}{\mathbb{Z}} 
\newcommand{\Q}{\mathbb{Q}}
\newcommand{\R}{\mathbb{R}}
\newcommand{\bC}{\mathbb{C}}

\newcommand{\Compacts}{\mathcal{K}}

\newcommand{\colim}{\operatorname{colim}}

\newcommand{\cA}{\mathcal{A}}
\newcommand{\cC}{\mathcal{C}}
\newcommand{\cJ}{\mathcal{J}}
\newcommand{\cF}{\mathcal{F}}

\newcommand{\topo}{\mathrm{top}}
\newcommand{\OrFG}{\mathrm{Or}_{\mathcal{F}}(G)}
\newcommand{\OrFING}{\mathrm{Or}_{\mathcal{FIN}}(G)}

\newcommand{\s}[1]{\langle #1 \rangle}

\makeatletter %\filcolim command for filtered colimits, write \filcolim_n A_n for example
\newcommand{\filcolim@}[2]{%
  \vtop{\m@th\ialign{##\cr
    \hfil$#1\operator@font colim$\hfil\cr
    \noalign{\nointerlineskip\kern1.5\ex@}#2\cr
    \noalign{\nointerlineskip\kern-\ex@}\cr}}%
}
\newcommand{\filcolim}{%
  \mathop{\mathpalette\filcolim@{\rightarrowfill@\textstyle}}\nmlimits@
}
\makeatother

\begin{document} 
\title{$K$-theory of non-commutative Bernoulli Shifts}

\author[S. Chakraborty]{Sayan Chakraborty}
\author[S. Echterhoff] {Siegfried Echterhoff}
\author[J. Kranz]{Julian Kranz}
\author[S. Nishikawa]{Shintaro Nishikawa}

\address{S.C.: Stat-Math unit, Indian Statistical Institute, 203 Barrackpore Trunk Road, Kolkata 700 108, India. \newline \texttt{sayan2008@gmail.com}}
\address{S.E \& J.K.\& S.N.: Mathematisches Institut, Fachbereich Mathematik und Informatik der Universit\"at M\"unster, Einsteinstrasse 62, 48149 M\"unster, Germany. \newline \texttt{echters@uni-muenster.de}, \texttt{j\textunderscore kran05@uni-muenster.de}, \texttt{snishika@uni-muenster.de}.} 

\thanks{This work was funded by the Deutsche Forschungsgemeinschaft (DFG, German Research Foundation) Project-ID 427320536 SFB 1442 and under Germany's Excellence Strategy EXC 2044  390685587, Mathematics M\"{u}nster: Dynamics, Geometry, Structure. The first author was also supported by  DST, Government of India under the \emph{DST-INSPIRE
Faculty Scheme} with Faculty Reg. No. IFA19-MA139.}

\date{\today}


\maketitle

\begin{abstract}
For a large class of $C^*$-algebras $A$, we  calculate the $K$-theory of reduced crossed products 
$A^{\otimes G}\rtimes_rG$ of Bernoulli shifts  by groups 
satisfying the Baum--Connes conjecture. In particular, we give explicit formulas for finite-dimensional $C^*$-algebras, UHF-algebras, rotation algebras, and several other examples. As an application, we obtain a formula for the $K$-theory of reduced $C^*$-algebras of wreath products $H\wr G$ for large classes of groups $H$ and $G$. 
Our methods use a generalization of techniques developed by the second named author together with Joachim Cuntz and Xin Li, and a trivialization theorem for finite group actions on UHF algebras developed in a companion paper by the third and fourth named authors.
\end{abstract}

\tableofcontents

\section{Introduction}\label{sec-intro}
Let $G$ be a countable discrete group and let $A$ be a separable unital $C^*$-algebra. Equip the infinite tensor product $A^{\otimes G}$ with the natural Bernoulli shift action (see Section~\ref{sec-general} for the definition). The objective of this paper is to compute the $K$-theory group of the reduced crossed product $A^{\otimes G}\rtimes_r G$ in as many cases as possible.

Building up on the work in \cite{CEL}, Xin Li \cite{XinLi} computed this when $A$ is a finite-dimensional $C^*$-algebra of the form $\bC\oplus \bigoplus_{1\leq k\leq N} M_{p_k}$, assuming the Baum--Connes conjecture with coefficients (referred to as \emph{BCC} below) for $G$ \cite{BCH}. His motivation was to compute the $K$-theory of the reduced group $C^*$-algebra of the wreath product $H\wr G$ for an arbitrary finite group $H$. For  $A=\bC\oplus \bigoplus_{1\leq k\leq N} M_{p_k}$, the $K$-theory groups $K_\ast\left(A^{\otimes G}\rtimes_rG\right)$ are computed in \cite{XinLi} as 
\begin{equation}\label{eq-XinLi}
K_\ast\left(A^{\otimes G}\rtimes_rG\right) \cong  \bigoplus_{[F]\in G\backslash \FIN} \bigoplus_{[S] \in G_F\backslash (\{1, \ldots, N\}^F) } K_\ast(C^*_r(G_S)) \\
\end{equation}
where $\FIN$ is the set of all the finite subsets of $G$ equipped with the left-translation $G$-action and $G_F:=\Stab_G(F)$, resp.~$G_S:=\Stab_G(S)$, are the stabilizer groups at $F$, resp.~at $S$,  for the action of $G$ on $\FIN$. Note that when $G$ is torsion-free, \eqref{eq-XinLi} is the direct sum of $K_\ast(C^*_r(G))$ and infinitely many copies of $K_\ast(\bC)$. For wreath products $H\wr G$ with respect to a finite group $H$, we have $C^*_r(H\wr G)\cong C^*_r(H)^{\otimes G}\rtimes_r G$. In this case, the number $N$ in \eqref{eq-XinLi} corresponds to the number of non-trivial conjugacy classes of $H$. 
Our results were also motivated in part by the paper \cite{Ohhashi} by Issei Ohhashi, where he gives 
$K$-theory computations for crossed products $A^{\otimes Z}\rtimes \Z$ for Bernoulli shifts by the integer group $\Z$. 

Our first result computes $K_\ast\left(A^{\otimes Z}\rtimes_rG\right)$ for an arbitrary finite-dimen\-sion\-al $C^*$-algebra $A$ and for an arbitrary infinite countable $G$-set $Z$ with the action of $G$ on $A^{\otimes Z}$ induced by the $G$-action on $Z$. In what follows we let $\FIN(Z)$ denote the collection of finite subsets of $Z$ and we put
 $\FINx(Z):=\FIN(Z)\setminus\{\emptyset\}$. In the case $Z=G$ we shall simply write $\FIN$ and $\FINx$ as above.

\begin{introtheorem}[{Theorem \ref{thm-finitedim-shift}}] \label{introthm-finite-dimensional} 
Let $G$ be a discrete group satisfying BCC and let $Z$ be a countably  infinite  $G$-set. Let $A=\bigoplus_{0\leq j \leq N}{M_{k_j}}$ where $k_0, \ldots, k_N$ ($N\geq1$) are positive integers with $\gcd(k_0, \ldots, k_N)=n$. We have an explicit isomorphism 
\begin{align*}
K_\ast\left(A^{\otimes Z}  \rtimes_rG\right)  \cong  \bigoplus_{[F]\in G\backslash \FIN(Z)} \bigoplus_{[S] \in G_F\backslash (\{1, \ldots, N\}^F) } K_\ast(C^*_r(G_S))[1/n]. \\
\end{align*}
For $N=0$, we have 
\[
K_\ast\left(M_n^{\otimes Z}  \rtimes_rG\right) \cong K_\ast(C^*_r(G))[1/n].
\]
\end{introtheorem}
Our strategy of the proof is as follows: First, we consider the case $n=1$ and construct a unit-preserving $KK$-equivalence $\bC^{N+1}\sim_{KK} A$ using an arithmetic argument about matrices in $\SL(n,\Z)$ with non-negative entries. By an argument borrowed from \cite{Izumi, Szabo}, this allows us to replace $A$ by $\bC^{N+1}$ and apply Xin Li's formula \eqref{eq-XinLi}. Second, we reduce the general case to the case $n=1$ with the help of a trivialization theorem for finite group actions on UHF algebras developed in the companion paper \cite{KN}.

One of our main techniques is the following theorem which is inspired by the \emph{regular basis} technique from \cite{CEL}.

\begin{introtheorem} [{Theorem \ref{thm-replace-by-JB-A0}, Corollary \ref{cor-UCT}}] \label{introthm-JB}
Let $G$ be a discrete group satisfying BCC. Let $A$ be a separable, unital $C^*$-algebra that satisfies the Universal Coefficient Theorem (UCT) and such that the unital inclusion $\iota\colon \bC \to A$ induces a split-injection $K_0(\bC) \to K_0(A)$.  Then $K_\ast\left(A^{\otimes G} \rtimes_rG\right)$ only depends on $G$ and the cokernel $\tilde K_\ast(A)$ of $\iota_\ast\colon K_\ast(\bC) \to K_\ast(A)$. 
For any countable $G$-set $Z$ we have
\begin{align*}
K_\ast\left(A^{\otimes Z} \rtimes_rG\right) & \cong \bigoplus_{[F]\in G\backslash \FIN(Z)}K_\ast\left(B^{\otimes F} \rtimes_rG_F\right),
\end{align*}
where $B$ is any $C^*$-algebra satisfying UCT with $K$-theory isomorphic to $\tilde K_\ast(A)$\footnote{Note that this determines $B$ uniquely up to $KK$-equivalence.} and where $G_F=\Stab_G(F)$. In particular, if $G$ is torsion-free and the action of $G$ on $Z$ is free, we have
\begin{align*}
K_\ast\left(A^{\otimes Z} \rtimes_rG\right) \cong K_\ast(C^*_r(G))\oplus \bigoplus_{[F]\in G\backslash \FINx(Z)}K_\ast\left(B^{\otimes F}\right).
\end{align*}

\end{introtheorem}

Note that the second formula can be computed more explicitly using the K\"unneth theorem. In the important case that  $Z=G$ with the left translation action (or, more generally, if $G$ acts properly on $Z$) the stabilizers $G_F$ 
for $F\in \FINx(Z)$ are all finite and the 
formula becomes more explicit once we can compute $K_\ast\left(B^{\otimes H}\rtimes_rH\right)$ for finite groups $H$ and for relevant building blocks for $B$ like $\bC$, $C_0(\R)$, the Cuntz-algebras $\cO_n$, and $C_0(\R)\otimes \cO_n$. For $H=\Z/2$, these computations have been done by Izumi \cite{Izumi}.
For general $H$, computing these $K$-groups is a non-trivial, indeed challenging task. For $B=C_0(\R)$ however, they are nothing but the equivariant topological $K$-theory $K^\ast_H\left(\R^H\right)$ for the $H$-Euclidean space $\R^{H}$, i.e.,  $\R^{|H|}$ with $H$-action induced by translation of coordinates. The groups $K^\ast_H\left(\R^H\right)$ are quite well-studied (\cite{Karoubi} \cite{EP}). Using these results, we give a more explicit formula for $A=C(S^1)$ (Example \ref{ex-S1}) and for the rotation algebras (or noncommutative tori) $A=A_\theta$ (Example \ref{ex-rotation}). More explicitly, we have

\begin{introtheorem}[{Example \ref{ex-S1}}]\label{introthm-S1} Let $G$ be a  discrete group satisfying BCC. We have
\begin{align*}
& K_\ast\left(C(S^1)^{\otimes G}\rtimes_rG\right) \\
 &\cong  K_\ast(C^*_r(G)) \oplus \left( \bigoplus_{\underset{|G_F\backslash F| \,\, \mathrm{even}}{[F]\in G\backslash \FIN^\times,}}K_\ast(C^*_r(G_F)) \right)  \oplus  \left(\bigoplus_{\underset{|G_F\backslash F| \,\, \mathrm{odd}}{[F]\in G\backslash \FIN^\times,}}K^\ast_{G_F}\left(\R^{G_F}\right) \right).
\end{align*}
\end{introtheorem}

As another application of Theorem \ref{introthm-JB}, we obtain a formula for the $K$-theory of reduced $C^*$-algebras of many wreath products $H\wr G$:

\begin{introtheorem}(Theorem \ref{thm-wreath})\label{introthm-wreath} Let $G$ be a discrete group satisfying BCC and let $H$ be a discrete group such that $C^*_r(H)$ satisfies the UCT and such that the inclusion $\bC\to C^*_r(H)$ induces a split injection on $K_0$. Then we have
\[
K_\ast(C^*_r(H\wr G)) \cong K_\ast(C^*_r(G)) \oplus \bigoplus_{[F] \in G\backslash \FINx} K_\ast\left(B^{\otimes F} \rtimes_rG_F\right),
\] where $B$ is any $C^*$-algebra satisfying the UCT with $K$-theory isomorphic to $\tilde K_*(C^*_r(H))$. In particular, if $G$ is torsion-free, we have
\begin{align*}
K_\ast(C^*_r(H\wr G)) \cong K_\ast(C^*_r(G))\oplus \bigoplus_{[F]\in G\backslash \FINx}K_\ast\left(B^{\otimes F}\right).
\end{align*}
\end{introtheorem}
We note that both the UCT assumption on $C_r^*(H)$ and the split-injectivity of the map $K_0(\bC)\to K_0(C_r^*(H))$ hold for every a-$T$-menable (in particular every amenable) group $H$ (see Remark \ref{rem-split-injective} below).

In Section \ref{sec-examples}, we also obtain formulas for several $C^*$-algebras $A$ that are not covered in Theorem \ref{introthm-JB}, in particular for Cuntz algebras.

\begin{introtheorem}[Corollary \ref{cor-On-localized-is-zero}, Proposition \ref{prop-On-fg}]\label{introthm-On}
	Let $G$ be a discrete group satisfying BCC and let $n\geq 2$. Then we have 
		\[K_*\left(\cO_{n+1}^{\otimes G}\rtimes_r G\right)[1/n]=0,\]
	where $\cO_{n+1}$ is the Cuntz algebra on $n+1$ generators.
	If $G$ is finite and $n=p$ prime, then $K_*\left(\cO_{p+1}^{\otimes G}\rtimes_r G\right)$ is a finitely generated $p$-group, that is a group of the form $\bigoplus_{1\leq j\leq N}\Z/p^{k_j}\Z$. 
\end{introtheorem}

In Section \ref{sec-AF}, we obtain a $K$-theory formula for Bernoulli shifts on unital AF-algebras in terms of colimits over the orbit category $\OrFING$. The result also applies to more general examples, in particular to many unital ASH-algebras (see Remark \ref{rem-AF-examples}).

\begin{introtheorem}[Theorem \ref{thm-AF-general}]\label{introthm-AF}
	Let $A$ be a unital AF-algebra, let $G$ be an infinite discrete group satisfying BCC, let $Z$ be a countable proper $G$-set, and let $S\subseteq \Z$ be the set of all positive integers $n$ such that $[1_A]\in K_0(A)$ is divisible by $n$. 
	Then, the natural inclusions

\begin{align*}
C^*_r(H) \to C^*_r(G),  \,\,\,& C^*_r(H)\to A^{\otimes Z}\rtimes_rH,\nonumber \\
C^*_r(G) \to A^{\otimes Z}\rtimes_rG, \,\,\,&  A^{\otimes Z}\rtimes_rH \to A^{\otimes Z}\rtimes_rG,
\end{align*}
induce the following pushout diagram 
\[
\xymatrix{
  \underset{G/H \in \OrFING}{\colim} K_\ast(C^*_r(H))[S^{-1}]    \ar[r] \ar[d] &  \underset{G/H \in \OrFING}{\colim} K_\ast\left(A^{\otimes Z} \rtimes_rH\right) \ar[d] \\
 K_\ast(C^*_r(G))[S^{-1}]   \ar[r]  & K_\ast\left(A^{\otimes Z} \rtimes_rG\right).
}
\] 
In particular, if $G$ is torsion-free, this pushout diagram reads
\[
\tilde K_\ast(C^*_r(G))[S^{-1}]  \oplus   K_\ast\left(A^{\otimes Z})_G \cong   K_\ast(A^{\otimes Z} \rtimes_rG\right),
\]
where $\tilde K_\ast(C^*_r(G))$ denotes the cokernel of the map $K_\ast(\bC)\to K_\ast(C_r^*(G))$ induced from the unital 
inclusion $\bC\hookrightarrow C_r^*(G)$.
\end{introtheorem}

The paper is structured as follows: In Section \ref{sec-general}, we develop our main machinery, including Theorem \ref{introthm-JB}. We apply this machinery in Section \ref{sec-fd} to prove Theorem \ref{introthm-finite-dimensional}.  In Section \ref{sec-examples}, we compute the $K$-theory of many more examples, including Theorems \ref{introthm-S1}, \ref{introthm-wreath}, and \ref{introthm-On}. 
Bernoulli shifts on unital AF-algebras are investigated in Section \ref{sec-AF} where we prove Theorem \ref{introthm-AF}. Based on similar ideas, we obtain in Section \ref{sec-modulo-k} some very general $K$-theory formulas 
up to inverting an integer $k$ or up to tensoring with the rationals $\Q$ (see Theorem \ref{thm-UCT-k} and Theorem \ref{thm-modulo-k-general}).
The rational $K$-theory computations apply to all unital stably finite $C^*$-algebras satisfying the UCT.

\subsection*{Acknowledgements}
The authors would like to thank Arthur Bartels and Eusebio Gardella for helpful discussions and N. Christopher Phillips for pointing out Corollary \ref{cor-Rokhlin} to them.

\section{General strategy}\label{sec-general}
To avoid technical complications with $KK$-theory, we assume throughout that all $C^*$-algebras are separable except for
the algebra $\mathcal B(H)$ of bounded operators on a Hilbert space $H$. For a $C^*$-algebra $A$ and a finite set $F$, we write $A^{\otimes F}$ to denote the minimal tensor product $\otimes_{x\in F}A$. If $A$ is moreover unital and $Z$ is a (not necessarily finite) countable set, we denote by $A^{\otimes Z}$ the filtered colimit\footnote{We use the term `filtered colimit' (which is the standard categorical notion) instead of the terms `direct limit' or `inductive limit' which seem to be  more commonly used in $C^*$-algebra theory. We do this to be consistent with the the use of more general colimits of functors in Section \ref{sec-AF}.} $\filcolim A^{\otimes F}$ taken over all finite subsets $F\subseteq Z$ with respect to the connecting maps $A^{\otimes F}\ni x\mapsto x\otimes 1\in A^{\otimes F}\otimes A^{\otimes F'- F}\cong A^{\otimes F'}$, for finite sets $F, F'$ with $F\subseteq F'.$ Hence $A^{\otimes Z}$ is the closed linear span of the elementary tensors $\bigotimes_{z \in Z} a_z$, where $a_z \in A$ for all $z \in Z,$ and $a_z=1$ for all but finitely many $z$. 
If $G$ is a discrete group acting on $Z$, we call the $G$-action on $A^{\otimes Z}$ given by permutation of the tensor factors the \emph{Bernoulli shift}. More explicitly, $g\bigg{(}\bigotimes_{z \in Z} a_s\bigg{)}=\bigotimes_{z \in Z} a_{g^{-1} z}$ for an elementary tensor $\bigotimes_{z \in Z} a_z \in A^{\otimes Z}$. 

For a countable discrete group $G$ and a $G$-$C^*$-algebra $A$, the \emph{Baum--Connes conjecture with coefficients} (\emph{BCC}) predicts a formula for the $K$-theory $K_*(A\rtimes_r G)$ of the reduced crossed product $A\rtimes_r G$, see \cite{BCH}. The precise formulation of the conjecture is not too important for us. We mostly need one of its consequences recalled in Theorem \ref{thm-goingdown} below. Note that the Baum--Connes conjecture with coefficients has been verified for many groups, including all a-$T$-menable groups \cite{HK} and all hyperbolic groups \cite{Lafforgue}.

We refer the reader to \cite{Kasparov} for the definition and basic properties of the equivariant $KK$-groups $KK^G(A,B)$. Recall from \cite{MN} that $KK^G$ can be organized into a triangulated category\footnote{Recently, $KK^G$ was even refined to a stable $\infty$-category \cite{BEL}.} with $G$-$C^*$-algebras as objects, the groups $KK^G(A,B)$ as morphism sets, and composition given by the Kasparov product. The construction of $KK^G$-elements from $G$-$*$-homomorphisms may be interpreted as a functor from the category of $G$-$C^*$-algebras to the category $KK^G$. Furthermore, $A\mapsto K_*(A\rtimes_r G)$ factorizes through this functor.

\begin{definition} A morphism $\phi$ in $KK^G(A, B)$ is a \emph{weak $K$-equivalence} in $KK^G$ if  its restrictions to $KK^H(A,B)$  induce isomorphisms  
\[
K_\ast(A\rtimes_rH) \cong K_\ast(B\rtimes_rH)
\]
for all finite subgroups $H$ of $G$. 
\end{definition} 

\begin{remark} Weak $K$-equivalences are in general weaker than weak equivalences in $KK^G$ in the sense of \cite{MN}. The latter ones are required to induce a $KK^H$-equivalence for all finite subgroups $H$ of $G$. 
\end{remark}

The following theorem has been shown in \cite{CEO} in the setting of locally compact groups (see also \cite{MN}). 
A detailed proof in the (easier) discrete  case is given in \cite[{Section 3.5}]{CELY}.

\begin{theorem}\label{thm-goingdown} Suppose that the Baum--Connes conjecture holds for $G$ with coefficients in $A$ and $B$. Then, any weak $K$-equivalence $\phi$ in $KK^G(A, B)$ induces an isomorphism
\[
K_\ast(A\rtimes_rG) \cong K_\ast(B\rtimes_rG).
\]
\end{theorem}

Let $A$ be a unital $C^*$-algebra and let $G$ be a countable group satisfying BCC. Our general strategy to compute $K_*(A^{\otimes Z}\rtimes_r G)$ is to replace $A^{\otimes Z}$ by a weakly $K$-equivalent $G$-$C^*$-algebra with computable $K$-theory and then apply Theorem \ref{thm-goingdown} above. 

The following lemma and corollary are straightforward generalizations of \cite[Theorem 2.1]{Izumi} and \cite[Corollary 6.9]{Szabo}.

\begin{lemma}[{see \cite[Theorem 2.1]{Izumi}}] \label{lem-Izumi} Let $A$ and $B$ be not necessarily unital $C^*$-algebras, let $G$ be a countable discrete group, and let $F$ be a finite $G$-set. Then there is a map from $KK(A, B)$ to $KK^G(A^{\otimes F}, B^{\otimes F})$ which sends the class of a $\ast$-homomorphism $\phi\colon A\to B$ to the class of $\phi^{\otimes F}$. Furthermore, this map is compatible with compositions (i.e. Kasparov products) and in particular sends $KK$-equivalences to a $KK^G$-equivalences. In particular, the Bernoulli shifts on $A^{\otimes F}$ and $B^{\otimes F}$ are $KK^G$-equivalent if $A$ and $B$ are $KK$-equivalent. The analogous statement holds if we replace the minimal tensor product $\otimes$ by the maximal one. \end{lemma}
\begin{proof} We recall the description of $KK^G$ in terms of asymptotic morphisms which admit an equivariant c.c.p. lift \cite{Thomsen} (see also Appendix in \cite{KS}). For any $G$-$C^*$-algebras $A_0$ and $B_0$, let $[[A_0, B_0]]^G_{\mathrm{cp}}$ be the set of homotopy equivalence classes of completely positive equivariant asymptotic homomorphisms from $A_0$ to $B_0$, see \cite[Section 2]{Thomsen}. The usual composition law for asymptotic homomorphisms restricts to $[[A_0, B_0]]^G_{\mathrm{cp}}$, \cite[Theorem 2]{Thomsen}. Let $S$ be the suspension functor $SA_0=C_0(\R)\otimes A_0$. Denote by $\tilde\Compacts_G$ the $C^*$-algebra of compact operators on $\ell^2(G\times \N)$. It is a $G$-$C^*$-algebra with respect to the regular representation on $\ell^2(G)$. The assignment $(A_0, B_0) \mapsto \widetilde {KK}^G(A_0, B_0) = [[SA_0\otimes \tilde\Compacts_G, SB_0\otimes \tilde\Compacts_G]]^G_{\mathrm{cp}}$, $(\phi\colon A_0\to B_0)\mapsto S\phi\otimes \mathrm{id}_{ \tilde\Compacts_G}$ defines a bifunctor from $G$-$C^*$-algebras to abelian groups.  Thomsen showed that there is a natural isomorphism $\widetilde {KK}^G(A_0, B_0)\cong KK^G(A_0, B_0)$ of bifunctors that sends the composition product of asymptotic morphisms to the Kasparov product, see \cite[Theorem 4.8]{Thomsen}. Now, given $\phi\in KK(A,B)$, we represent $\phi$ as a completely positive asymptotic homomorphism from $SA\otimes \Compacts(\ell^2(\N))$ to $SB\otimes \Compacts(\ell^2(\N))$. Taking the (pointwise) minimal tensor product of $\phi$ with itself over $F$, we obtain a completely positive equivariant asymptotic homomorphism $\phi^{\otimes F}$ from $(SA\otimes \Compacts(\ell^2(\N)))^{\otimes F}$ to $(SB\otimes \Compacts(\ell^2(\N)))^{\otimes F}$. This construction clearly respects homotopy equivalences. Therefore, $[\phi]\mapsto [\phi^{\otimes F}]$ defines a map $KK(A, B)\to KK^G((SA\otimes \Compacts(\ell^2(\N)))^{\otimes F}, (SB\otimes \Compacts(\ell^2(\N)))^{\otimes F})$. By construction, this map is compatible with compositions and sends a $\ast$-homomorphism $\phi\colon A\to B$ to $(S\phi \otimes \mathrm{id}_{\Compacts(\ell^2(\N))})^{\otimes F}$. By the stabilization theorem and by Kasparov's Bott-periodicity theorem\footnote{Since the $G$-action on $F$ factors through the finite group $\mathrm{Sym}(F)$, the theorem is applicable even if $G$ itself is not finite.} (see \cite[Theorem 20.3.2]{Blackadar}), the exterior tensor product by the identity on $C_0(\R)^{\otimes F}\otimes \Compacts(\ell^2(\N))^{\otimes F}$ induces a natural isomorphism
\begin{multline*}
  KK^G(A_0, B_0) \\  \cong KK^G(A_0\otimes C_0(\R)^{\otimes F}\otimes \Compacts(\ell^2(\N))^{\otimes F}, B_0\otimes C_0(\R)^{\otimes F}\otimes \Compacts(\ell^2(\N))^{\otimes F} )
  \end{multline*}
  for all $G$-$C^*$-algebras $A_0$ and $B_0$. Therefore, the map $[\phi]\mapsto [\phi^{\otimes F}]$ can be naturally regarded as a map from $KK(A, B)$ to $KK^G(A^{\otimes F}, B^{\otimes F})$. By construction, this map is compatible with compositions and sends the class of a $\ast$-homomorphism $\phi\colon A\to B$ to the class of $\phi^{\otimes F}$. The case of the maximal tensor product is proven in the exact same way.
  \end{proof}

We emphasize that the map $\phi\mapsto \phi^{\otimes F}$ is \emph{not} a group homomorphism. For example, it maps $n$ in $KK(\bC, \bC)\cong \Z$ to the class $[\pi_n]$ in $KK^G(\bC, \bC)$ of the finite dimensional unitary representation $\pi_n$ of $G$ on $\ell^2(\{1, \ldots, n\}^F)$ defined by permutation on the set $\{1, \ldots, n\}^F$.

\begin{corollary}[{see \cite[Corollary 6.9]{Szabo}}]\label{cor-weak-equivalence} 
	Let $G$ be a discrete group and let $\phi\colon A\to B$ be a unital $*$-homomorphism which is a $KK$-equivalence. Then for any countable $G$-set $Z$, the map 
		\[\phi^{\otimes Z}\colon A^{\otimes Z}\to B^{\otimes Z}\]
	is a weak $K$-equivalence in $KK^G$. In particular, $\phi^{\otimes Z}$ induces an isomorphism 
	\[K_*(A^{\otimes Z}\rtimes_r G)\to K_*(B^{\otimes Z}\rtimes_r G)\]
	whenever $G$ satisfies BCC. 
\end{corollary}
\begin{proof}
	For the first statement we may assume that $G$ is finite. Since $K(-\rtimes_r G)$ commutes with filtered colimits, we may as well assume that $Z$ is finite and apply Lemma \ref{lem-Izumi}. The second statement follows from Theorem \ref{thm-goingdown}.
\end{proof}

\begin{definition}\label{defn-JB-A0}
Let $G$ be a discrete group, let $Z$ be a countable $G$-set and let $A_0$ and $B$ be $C^*$-algebras with $A_0$ unital. We define a $G$-$C^*$-algebra $\cJ^Z_{A_0,B}$ as 
\begin{equation}\label{eq-cJ}
\cJ^Z_{A_0,B} \coloneqq \bigoplus_{F\in \FIN(Z)} A_0^{\otimes Z-F} \otimes B^{\otimes F}.
\end{equation}
The $G$-action is defined so that a group element $g\in G$ sends $A_0^{\otimes Z-F} \otimes B^{\otimes F}$ to $A_0^{\otimes Z-gF} \otimes B^{\otimes gF}$ by the obvious $\ast$-homomorphism. 
\end{definition}
In the above definition, $A_0$ should be thought of as either the complex numbers $A_0=\bC$ or a UHF-algebra (see Section \ref{sec-modulo-k}). For $A_0=\bC$, we just write $\cJ^Z_B\coloneqq \cJ^Z_{\bC,B}$. 
Note that $\cJ^Z_{A_0,B}$ is a $G$-$C_0(\FIN(Z))$-algebra in a natural way. The following lemma is crucial for our computations:

\begin{lemma}\label{lem-JB1} 
Let $G,Z,A_0$ and $B$ be as above. Then $\cJ^Z_{A_0,B}$ can naturally be identified with the filtered colimit $\filcolim_S A_0^{\otimes Z-S}\otimes (A_0\oplus B)^{\otimes S}$ taken over all finite subsets $S\subseteq Z$, with respect to the obvious connecting maps.\footnote{If $S\subseteq S'$ the connecting map is 
the tensor product of the identity on  $(A_0\oplus B)^{\otimes S}$ with the canonical inclusion 
$A_0^{\otimes Z-S}=A_0^{\otimes Z-S'}\otimes A_0^{\otimes S'-S}\hookrightarrow A_0^{\otimes Z-S'}\otimes (A_0\oplus B)^{\otimes S'-S}$.}
\end{lemma}
\begin{proof}
Just observe that for all finite subsets $S\subseteq Z$ we have canonical isomorphisms
\[
A_0^{\otimes Z-S}\otimes (A_0\oplus B)^{\otimes S}\cong 
A_0^{\otimes Z-S}\otimes \bigg{(}\bigoplus_{F\subseteq S} A_0^{\otimes S-F}\otimes B^{\otimes F}\bigg{)}
\cong \bigoplus_{F\subseteq S} A_0^{\otimes Z-F}\otimes B^{\otimes F}.
\]
\end{proof}

\begin{theorem}[Theorem \ref{introthm-JB}]\label{thm-replace-by-JB-A0}
	Let $A$, $A_0$ and $B$ be $C^*$-algebras with $A$ and $A_0$ unital and let $\iota\colon A_0\to A$ be a unital $\ast$-homomorphism. Let $\phi\in KK(B, A)$ be an element such that $\iota\oplus \phi\in KK(A_0 \oplus B,A)$ is a $KK$-equivalence. Then for each countable $G$-set $Z$, there is a weak $K$-equivalence 
in $KK^G(\cJ^Z_{A_0,B},A^{\otimes Z})$.
	If $G$ moreover satisfies BCC, there is an isomorphism
 \begin{multline*}
 K_*\left(A^{\otimes Z}\rtimes_r G\right)\cong K_*\left(\cJ^Z_{A_0,B}\rtimes_r G\right)\\ \cong \bigoplus_{[F]\in G\backslash \FIN(Z)}K_*\left(\left(A_0^{\otimes Z-F}\otimes B^{\otimes F}\right)\rtimes_r G_F\right),
 \end{multline*}
	where $\FIN(Z)$ denotes the set of finite subsets of $Z$ and where $G_F$ denotes the stabilizer of $F$ for the action of $G$ on $\FIN(Z)$. 
\end{theorem} 
For the proof, we need the following well-known lemma:

\begin{lemma}\label{lem-KKproduct} 
Let $Z$ be a countable $G$-set and let $A_Z$ be a $G$-$C_0(Z)$-algebra. For any $G$-$C^*$-algebra $D$ and any choice of representatives $z$ for $[z]\in G\setminus Z$, there is a natural isomorphism
\[
\Psi=(\Psi_z)_{[z]}\colon KK^G(A_Z, D) \xrightarrow{\cong} \prod_{[z] \in G\backslash Z}KK^{G_z}(A_z, D)
\]
where $G_z$ is the stabilizer of $z$ and $A_z$ is the fiber of $A_Z$ at $z\in Z$. 
The map $\Psi$ is independent of the choice of representatives in the sense that for any $z\in Z$ and $g\in G$, the diagram 
\[\xymatrix{
	KK^G(A_Z, D)\ar[r]^{\Psi_z}\ar[dr]_{\Psi_{gz}}	&KK^{G_z}(A_z, D)\ar[d]_\cong^g\\
								&KK^{G_{gz}}(A_{gz},D)
}\]
commutes where $g\colon KK^{G_z}(A_z, D)\to KK^{G_{gz}}(A_{gz},D)$ is given by conjugation with $g$. 

\end{lemma}
\begin{proof}
Recall that for a subgroup $H\subseteq G$, and an $H$-$C^*$-algebra $B$, the induced $G$-$C^*$-algebra $\mathrm{Ind}_H^GB$ is defined as 
$$\operatorname{Ind}_H^G B:=\left\{f \in C_b(G, B)\,\middle\vert\, \begin{array}{c}h(f(s))=f(sh^{-1}), s\in G, h\in H \\ \text { and } s H \mapsto\|f(s)\| \in C_0(G / H)\end{array}\right\}$$
 equipped with the left-translation $G$-action (see \cite[Section 2]{CE} for example).
Since $Z$ is discrete, we have a natural isomorphism 
\begin{equation}\label{eq-decompose-induced}
A_Z = \bigoplus_{z\in Z}A_z \cong \bigoplus_{[z]\in G\backslash Z} \mathrm{Ind}_{G_z}^{G}A_z,
\end{equation}
where the last isomorphism follows from \cite[Theorem 3.4.13]{CELY}. 
Now let
	\[\Psi=(\Psi_z)_{[z]}\colon KK^G(A_Z,D)\to \prod_{[z]\in G\backslash Z}KK^{G_z}(A_z,D)\]
be the map with components
	\[\Psi_z\colon KK^G(A_Z,D)\xrightarrow{\mathrm{Res}_G^{G_z}} KK^{G_z}(A_Z,D)\xrightarrow{\iota_z^*} KK^{G_z}(A_z,D),\]
where $\iota_z\colon A_z\hookrightarrow A_Z$ is the inclusion. The second part of the lemma follows from the construction of $\Psi$ and the fact that conjugation by any $g\in G$ acts trivially on $KK^G(A_Z,D)$. We show that $\Psi$ is an isomorphism. 
By \eqref{eq-decompose-induced} and \cite[Theorem 2.9]{Kasparov}, we can identify $\Psi$ with the product, taken over all $[z]\in G\backslash Z$, of the compression maps
	\[\mathrm{comp}_{G_z}^G\colon KK^G(\mathrm{Ind}_{G_z}^GA_z,D)\to KK^{G_z}(A_z,D).\]
These are isomorphisms by \cite[Proposition 5.14]{CE} (see also \cite[(20)]{MN}). 
\end{proof}

\begin{proof}[Proof of Theorem \ref{thm-replace-by-JB-A0}]
For each orbit $[F]$ in $G \backslash \FIN(Z)$ with stabilizer $G_F$, we define an element 
\[
\Phi_{F}\in KK^{G_F}(A_0^{\otimes Z-F}\otimes B^{\otimes F}, A^{\otimes Z})
\]
as the composition 
\[
A_0^{\otimes Z-F}\otimes B^{\otimes F} \to A_0^{\otimes Z-F}\otimes(A_0 \oplus B)^{\otimes F} \to  A_0^{\otimes Z-F}\otimes A^{\otimes F} \to  A^{\otimes Z}.
\]
Here the first and the third map are the obvious maps and the second map is the tensor product of the identity on $A_0^{\otimes Z-F}$ and the $KK^{G_F}$-equivalence $(\iota \oplus \phi)^{\otimes F}$ obtained in Lemma \ref{lem-Izumi}.
When $F$ is the empty set, we define $\Phi_{F}$ as the unital map $\iota^{\otimes Z}\colon A_0^{\otimes Z} \to A^{\otimes Z}$. 
By Lemma \ref{lem-KKproduct} the family $\{\Phi_F: [F]\in G\backslash\FIN(Z)\}$ defines an element $\Phi\in KK^G(J_{A_0,B}^Z, A^{\otimes Z})$ that does not depend on the choice of representatives $F$ for each $[F]$ (since we have $\Phi_{gF}=g(\Phi_F)$ for every $g\in G$).

We show that $\Phi$ is a weak $K$-equivalence. By construction, for any finite subgroup $H$ of $G$ and for any finite $H$-subset $S$ of $Z$, the element $\Phi$ may be restricted to
\[
\Phi_S\colon  \bigoplus_{F\in \FIN(Z), F\subset S} A_0^{\otimes Z-F}\otimes B^{\otimes F} \to A_0^{\otimes Z-S}\otimes A^{\otimes S}
\]
in $KK^H$. Using Lemma \ref{lem-JB1}, $\Phi_S$ can be identified with the tensor product of the identity on $A_0^{\otimes Z-S}$ and the $KK^H$-equivalence $(\iota \oplus \phi)^{\otimes S} \in KK^H((A_0 \oplus B)^{\otimes S}, A^{\otimes S})$ via the isomorphism 
\[(A_0\oplus B)^{\otimes S}\cong \bigoplus_{F\in \FIN(Z), F\subset S}A_0^{\otimes S-F}\otimes B^{\otimes F}.\]
Thus each $\Phi_S$ induces an isomorphism 
\[K_*\left(\left(A_0^{\otimes Z-S}\otimes(A_0\oplus B)^{\otimes S}\right)\rtimes H\right)\cong K_*\left(\left(A_0^{\otimes Z-S}\otimes A^{\otimes S}\right)\rtimes H\right).\] 
By taking the filtered colimit over finite $H$-subsets $S\subseteq Z$, we see that $\Phi$ is a weak $K$-equivalence. 

The $K$-theory computation follows from the fact that 
$\cJ^Z_{A_0,B}$ is a $G$-$C_0(\FIN(Z))$-algebra for the discrete $G$-set $\FIN(Z)$ together with Green's imprimitivity Theorem (e.g., see \cite[Remark 3.13]{CEL} for details).
\end{proof}

The following special case of Theorem \ref{thm-replace-by-JB-A0} will be used to compute examples in Section \ref{sec-examples}. Its main advantage is that we do not have to construct the element $\phi$ in order to apply it. 
\begin{corollary}\label{cor-UCT}
	Let  $G$ be a discrete group satisfying BCC. Let $Z$ be a countable $G$-set and let $A$ be a unital $C^*$-algebra satisfying the UCT such that the unital inclusion $\bC\to A$ induces a split injection $K_*(\bC)\to K_*(A)$. Denote by $\tilde K_*(A)$ its cokernel and let $B$ be any $C^*$-algebra satisfying the UCT with $K_*(B)\cong \tilde K_*(A)$. Then we have 
	\[K_*\left(A^{\otimes Z}\rtimes_r G\right)\cong K_*(\cJ_B^Z\rtimes_r G)\cong \bigoplus_{[F]\in G\backslash \FIN(Z)}K_*\left(B^{\otimes F}\rtimes_r G_F\right).\]
	In particular, if $Z=G$ equipped with the left translation, then 
	$K_*\left(A^{\otimes G}\rtimes_r G\right)$ only depends on $G$ and $\tilde K_*(A)$. 
\end{corollary}
Note that $B$ exists and is uniquely determined up to $KK$-equivalence (see \cite[Corollary 23.10.2]{Blackadar}).
\begin{proof}
	By the UCT, the inclusion map $K_*(B)\cong \tilde K_*(A)\to K_*(A)$ is induced by an element $\phi\in KK(B,A)$. By construction, $\iota\oplus\phi\in KK(\bC\oplus B,A)$ is a $KK$-equivalence, so that we can apply Theorem \ref{thm-replace-by-JB-A0}.
\end{proof}

As another special case of Theorem \ref{thm-replace-by-JB-A0}, we recover
\begin{corollary}[{\cite[Example 3.17]{CEL}, \cite[Proposition 2.4]{XinLi}}]\label{cor-Cantor}Let $G$ be a discrete group satisfying BCC and let $Z$ be a countable $G$-set. Then for $n\geq1$, we have 
\[
K_\ast (C(\{0, 1, \ldots, n \}^Z)  \rtimes_rG) 
\cong  \bigoplus_{[F]\in G\backslash \FIN(Z)} \bigoplus_{[S] \in G_F\backslash (\{1, \ldots, n\}^F) } K_\ast(C^*_r(G_S)).\]
Moreover, if $Z=G$ with the left translation action, we get  
\begin{align*}
K_\ast(&C(\{0, 1, \ldots, n \}^G)  \rtimes_rG) \\
&\cong  K_\ast(C^*_r(G)) \oplus \bigoplus_{[C]\in \cC} \bigoplus_{[X] \in N_C \backslash F(C) }  \bigoplus_{[S] \in C \backslash \{1, \ldots, n\}^{C\cdot X} } K_\ast(C^*_r(C_S)).
\end{align*}
Here $\mathcal C$ denotes the set of all conjugacy classes of finite subgroups of $G$, $F(C)$  the nonempty finite subsets of $C\backslash G$, $N_C=\{g\in G: gCg^{-1}=C\}$ the normalizer of $C$ in $G$, and 
$C_S=G_S\cap C$ the stabilizer 
of $S$ in $C$.
\end{corollary}

\begin{proof}
	Let $A=C(\{0,\dotsc,n\})$ and $B=C(\{1,\dotsc,n\})$ and let $\phi\colon B\to A$ be the canonical inclusion. The first isomorphism follows from Theorem \ref{thm-replace-by-JB-A0}. The second isomorphism is obtained by analyzing the orbit structure of $\bigsqcup_{F\in \FIN(G)}\{1,\dotsc,n\}^F$ (see \cite[Proposition 2.4]{XinLi} for details). 
\end{proof}

\section{Finite-dimensional algebras}\label{sec-fd}
In this section we compute the $K$-theory of crossed products of the form $A^{\otimes Z}\rtimes_r G$ where $A=\bigoplus_{0\leq j \leq N}{M_{k_j}}$ is a finite-dimensional $C^*$-algebra, $Z$ is a countable $G$-set,  and where $G$ is a group satisfying BCC. This generalizes the case $k_0=1$ from \cite{XinLi}. We denote by $\gcd(k_0,\dotsc,k_N)$ the greatest common divisor of $k_0,\dotsc,k_N$.
We believe that the following theorem is known to experts. In lack of a reference, we give a detailed proof here.

\begin{theorem}\label{thm-coprime-KK}
	Let $k_0,\dotsc,k_N$ be positive integers with $\gcd(k_0,\dotsc,k_N)=1$. Then there is a unital $*$-homomorphism 
		\[\phi\colon \bC^{N+1}\to \bigoplus_{0\leq j \leq N}{M_{k_j}}\]
	that induces a $KK$-equivalence. Moreover when $N=1$, there are exactly two such $\ast$-homomorphisms up to unitary equivalence.
\end{theorem}
The main ingredient for the proof is the following arithmetic fact: 
\begin{proposition} \label{prop-sl2z} For any pair of positive integers $k_1,k_2\in \N$, there is a unique matrix $X$ in $\SL(2, \Z)$ with non-negative entries such that
\[
X \begin{bmatrix} n \\ n \end{bmatrix} = \begin{bmatrix} k_1 \\ k_2 \end{bmatrix} 
\]
where $n=\gcd(k_1,k_2)$. 
If we allow $X$ to be in $\GL(2, \Z)$, then there are exactly two such $X$: the one $X_0$ in $\SL(2, \Z)$ and $X_0 \begin{bmatrix}  0 & 1 \\ 1 & 0 \end{bmatrix}$.
\end{proposition} 

\begin{proof}
Existence: Let $f\colon\mathbb{N}_+^2\to \mathbb{N}_+^2$ be given by
\[
 \begin{bmatrix} k_1\\k_2\end{bmatrix} 
 	\mapsto 
 		\begin{cases}
 			\begin{bmatrix}k_1-k_2\\ k_2\end{bmatrix}, & \text { if } k_1 > k_2 \\[3ex]
 		\begin{bmatrix}k_1\\ k_2-k_1\end{bmatrix}, & \text { if } k_1< k_2 \\[3ex]
 		\begin{bmatrix}k_1\\ k_2\end{bmatrix}, & \text{ if } k_1=k_2.							\end{cases}
\]
The Euclidian algorithm precisely says that there exists a $k \in \mathbb{N}$ such that $f^k\begin{bmatrix}k_1\\ k_2\end{bmatrix}=\begin{bmatrix}n\\ n\end{bmatrix}.$ Let $A=\begin{bmatrix}1 & 1 \\ 0 & 1\end{bmatrix}$ and $B=\begin{bmatrix}1 & 0 \\ 1 & 1\end{bmatrix}.$ Then the map $f$ defined above is given by
\[
\begin{bmatrix}
k_1 \\
k_2
\end{bmatrix} \mapsto \begin{cases}A^{-1} \begin{bmatrix}
k_1 \\
k_2
\end{bmatrix}, & \text { if } k_1 > k_2 \\[3ex]
B^{-1} \begin{bmatrix}
k_1 \\
k_2
\end{bmatrix}, & \text { if } k_1< k_2 \\[3ex]
\begin{bmatrix}
k_1\\
k_2
\end{bmatrix}, & \text { if } k_1=k_2.
\end{cases}
\]
Now the above formulation of the Euclidean algorithm gives us integers $a_1,a_2,\cdots,a_l, b_1,b_2,\cdots,b_l\geq 1$
such that 
\[A^{a_1} B^{b_1} A^{a_2} B^{b_2} \cdots A^{a_l} B^{b_l}\begin{bmatrix}
n \\
n
\end{bmatrix}=\begin{bmatrix}
k_1 \\
k_2
\end{bmatrix}.\]
Then $X\coloneqq ¸A^{a_1} B^{b_1} A^{a_2} B^{b_2} \cdots A^{a_l} B^{b_l}\in \mathrm{SL}(2,\Z)$ is the required matrix.\\

Uniqueness: Suppose there are two such matrices $X_1$ and $X_2$. Then, $Y=X_2^{-1}X_1\in \SL(2, \Z)$ satisfies 
\[
Y \begin{bmatrix} 1 \\ 1 \end{bmatrix} = \begin{bmatrix} 1 \\ 1 \end{bmatrix}. 
\]
From this, we have
\[
Y= \begin{bmatrix} a & 1-a \\ a-1 &  2-a \end{bmatrix}
\]
for $a\in \Z$. Now, we get
\[
X_1= X_2 \begin{bmatrix} a & 1-a \\ a-1 &  2-a \end{bmatrix},
\]
but since both $X_1$ and $X_2$ have non-negative entries and since they are non-singular, it is not hard to see that $a$ must be $1$. Thus, $X_1=X_2$. The last assertion is immediate.
\end{proof}

\begin{remark} It follows from the proof that the subsemigroup in $\SL(2, \Z)$ consisting of matrices with non-negative entries is the free monoid of two generators $A$ and $B$.
\end{remark}

\begin{corollary}\label{cor-slnz}  For any positive integers $k_0, \ldots, k_N$ with $\gcd(k_0, \ldots, k_N)=n$, there is a matrix $X$ in $\SL(N+1, \Z)$ with non-negative entries such that
\[
X \begin{bmatrix} n \\  \vdots \\ n \end{bmatrix} = \begin{bmatrix} k_0 \\ \vdots \\ k_N \end{bmatrix}. 
\]
\end{corollary} 
\begin{proof} We give the proof by induction on $N$. The case $N=0$ is clear. Let $N\geq 1$ and let
$k_0,\ldots, k_N\geq 1$ be positive integers with $\gcd(k_0,\ldots, k_N)=n$. Set $l\coloneqq\gcd(k_0,\ldots, k_{N-1})$. By induction, 
we may assume that there exists a matrix $\tilde{X}\in \SL(N,\Z)$ with non-negative entries such that 
${\tilde{X}\begin{bmatrix} l\\ \vdots \\ l\end{bmatrix}=\begin{bmatrix} k_0\\ \vdots \\ k_{N-1}\end{bmatrix}}$. 
Since $\gcd(l, k_N)=\gcd(l,n)=n$, it follows from Proposition \ref{prop-sl2z} that there are matrices $X_0, X_{N}\in \SL(2,\Z)$ with non-negative entries such that\linebreak   ${X_0\begin{bmatrix} n\\n\end{bmatrix}=\begin{bmatrix} l\\n\end{bmatrix}}$
and $X_{N}\begin{bmatrix} n\\n\end{bmatrix}=\begin{bmatrix} l\\k_N\end{bmatrix}$. 
Now for $i\in \{1,\ldots, N-1\}$,
let $Y_i$ denote the matrix with $X_0$ as the $(i, i+1)$-th diagonal block, with ones in all other diagonal entries and with zeros elsewhere, and let $Y_{N}=\begin{bmatrix} I_{N-1} &0\\ 0& X_N\end{bmatrix}$. Write $Y\coloneqq\begin{bmatrix} \tilde{X}&0\\0&1\end{bmatrix}$. 
Then $X\coloneqq Y\cdot Y_N\cdots Y_1\in SL(N+1,\Z)$ has non-negative entries and 
\[ X\begin{bmatrix} n\\ \vdots \\n\end{bmatrix}=Y\left(Y_N\cdots Y_1\begin{bmatrix} n\\ \vdots \\n\end{bmatrix}\right)=Y\begin{bmatrix} l\\ \vdots \\l \\ k_N\end{bmatrix}=
\begin{bmatrix} k_0\\ \vdots\\ k_N\end{bmatrix}\]
as desired. 
\end{proof}

\begin{proof}[Proof of Theorem \ref{thm-coprime-KK}]
 Unital $\ast$-homomorphisms from $\bC^{N+1}$ to $\bigoplus_{0\leq j \leq N}{M_{k_j}}$ are classified up to unitary equivalence by their induced maps on ordered $K_0$-groups with units. 
 If we identify the $K_0$-groups of $\bC^{N+1}$ and $\bigoplus_{0\leq j \leq N}{M_{k_j}}$ with $\Z^{N+1}$, we may represent the induced maps on $K_0$-groups by 
 $(N+1)$-square matrices with non-negative integer entries that send 
 $\begin{bmatrix} 1 \\ \vdots \\ 1 \end{bmatrix}$ to $\begin{bmatrix} k_0 \\ \vdots \\ k_N \end{bmatrix}$. Such a matrix is an isomorphism on $K$-theory if and only if it is in  $\GL(N+1, \Z)$. All the assertions now follow from Proposition \ref{prop-sl2z} and Corollary \ref{cor-slnz}.
\end{proof}

We are now ready to prove Theorem \ref{introthm-finite-dimensional}. For the proof, we need a result from \cite{KN}. We formulate it in full generality here since this will be needed in later sections of the paper. 
Recall that a \emph{supernatural number} is a formal product $\mathfrak n=\prod_p p^{n_p}$ of prime powers with $n_p\in \{0,1,\dotsc,\infty\}$. We denote by $M_\mathfrak n\coloneqq \otimes_p M_p^{\otimes n_p}$ the associated \emph{UHF-algebra}. $M_\mathfrak n$ and $\mathfrak n$ are called \emph{of infinite type} if $n_p\in \{0,\infty\}$ for all $p$ and $n_p\neq 0$ for at least one $p$. 
For an abelian group $L$, we denote by 
	\[L[1/\mathfrak n]\coloneqq \colim(L\xrightarrow{\cdot p_1}L\xrightarrow{\cdot p_2}\cdots)\]
the \emph{localization at $\mathfrak n$} where $(p_1,p_2,\dotsc)$ is a sequence of primes containing every $p$ with $n_p\geq 1$ infinitely many times. In other words, $L[1/\mathfrak n]$ is the localization at the set of all primes dividing $\mathfrak n$. If $\mathfrak n$ is finite, this definition recovers the usual localization at a natural number. 

\begin{theorem}[{\cite[Corollary 2.11]{KN}}]\label{thm-localize}
Let $G$ be a discrete group satisfying BCC, let $Z$ a countable $G$-set, let $A$ a $G$-$C^*$-algebra and let $M_\mathfrak n$ a UHF-algebra. Assume that $Z$ is infinite or that $\mathfrak n$ is of infinite type. Then the inclusion $A\to A\otimes M_\mathfrak n^{\otimes Z}$ induces an isomorphism
\[
K_*(A \rtimes_rG)[1/\mathfrak n]\cong K_*\left(\left(A \otimes M_\mathfrak n^{\otimes Z}\right)\rtimes_r G\right).
\]
In particular, the right-hand side is a $\Z[1/\mathfrak n]$-module. 
\end{theorem}

\begin{theorem}[Theorem \ref{introthm-finite-dimensional}] \label{thm-finitedim-shift}
Let $G$ be a discrete group satisfying BCC. Let 
$Z$ be a countably infinite $G$-set
and let $A=\bigoplus_{0\leq j \leq N}{M_{k_j}}$ where $k_0, \ldots, k_N$ ($N\geq1$) are positive integers with 
$\gcd(k_0, \ldots, k_N)=n$. Then 
\begin{align*}
K_\ast\left(A^{\otimes Z}  \rtimes_rG\right) & \cong  K_\ast\left(C\left(\{0, \ldots, N \}^Z \right) \rtimes_rG\right)[1/n] \\
& \cong  \bigoplus_{[F]\in G\backslash \FIN(Z)} \bigoplus_{[S] \in G_F\backslash (\{1, \ldots, N\}^F) } K_\ast(C^*_r(G_S))[1/n]. 
\end{align*}
\end{theorem}

\begin{proof}
	Write $B=\bigoplus_{0\leq j \leq N}{M_{k_j/n}}$ so that $B$ satisfies the assumptions of Theorem \ref{thm-coprime-KK} and so that $A\cong B\otimes M_n$. By Theorem \ref{thm-localize}, the inclusion $B^{\otimes Z}\hookrightarrow A^{\otimes Z}$ induces an isomorphism 
		\[K_*\left(B^{\otimes Z}\rtimes_r G\right)[1/n]\cong K_*\left(A^{\otimes Z}\rtimes_r G\right).\]
	By Theorem \ref{thm-coprime-KK} and Corollary \ref{cor-weak-equivalence}, we furthermore have 
		\[K_*\left(B^{\otimes Z}\rtimes_r G\right)\cong K_\ast\left(C\left(\{0, \ldots, N \}^Z \right) \rtimes_rG\right).\]
	This proves the isomorphism in the first line of the theorem. The isomorphism in the second line follows from Corollary \ref{cor-Cantor}. 
	\end{proof}

\section{More examples}\label{sec-examples}
In this section, we compute the $K$-theory of Bernoulli shifts in more examples.  We mostly restrict ourselves to the case $Z=G$ with the left translation action, but some of the results have straightforward generalizations to arbitrary countable $G$-sets. Recall that for $Z=G$ we write $\FIN=\FIN(G)$ for the collection of finite subsets of $G$ and we put $\cJ_B:=\cJ_{\bC,B}^G$  as in Definition \ref{defn-JB-A0} for any $C^*$-algebra $B$. 
We start with some easy applications of Corollary \ref{cor-weak-equivalence}.

\begin{example}\label{ex-Fibonacci-compact} 
	Let $\cA$ be the Fibonacci algebra \cite[Example III 2.6]{Davidson}, which is the filtered colimit of $(M_{m_k}\oplus M_{n_k})_{k\in \N}$ where $m_1=n_1=1$ and where the connecting maps are given by repeated use of the partial embedding matrix $\begin{bmatrix} 1 & 1 \\ 1 & 0 \end{bmatrix}$. Since this matrix belongs to $\GL(2,\Z)$, the unital embedding $\bC\oplus \bC\hookrightarrow \cA$ is a $KK$-equivalence. 
	
	Let $\mathcal{K}^+$ be the unitization of the algebra of compact operators on $\ell^2(\N)$ and let $p\in \mathcal K$ be a rank-$1$ projection. Then the unital embedding ${\bC\oplus \bC\hookrightarrow \mathcal K^+}$ given by $(\lambda,\mu)\mapsto \lambda(1-p)+\mu p$ is a $KK$-equivalence. 
	
	Now let $A$ be either $\cA$ or $\mathcal K^+$ and let $G$ be a discrete group satisfying BCC. Then by Corollaries \ref{cor-weak-equivalence} and \ref{cor-Cantor}, we have 
	\[
K_\ast\left(A^{\otimes G} \rtimes_rG\right) \cong K_\ast(C^*_r(G)) \oplus \bigoplus_{[F]\in G\backslash \FINx} K_\ast( C^*_r(G_F)).
\]
In particular, if $G$ is torsion free, we have
\[
K_\ast\left(A^{\otimes G} \rtimes_rG\right) \cong K_\ast(C^*_r(G)) \oplus \bigoplus_{[F]\in G\backslash \FINx} K_\ast(\bC).
\]	
\end{example}

\begin{example}[Theorem \ref{introthm-S1}]\label{ex-S1} 
Consider $A=C(S^1)$. 
Note that the canonical inclusion ${\phi\colon C_0(\R) \to C(S^1)}$ together with the unital inclusion 
 $\iota\colon \bC \to C(S^1)$ induces a $KK$-equivalence $\iota\oplus \phi \in KK(\bC\oplus C_0(\R), A)$. 
 By Theorem \ref{thm-replace-by-JB-A0}, we obtain a weak $K$-equivalence 
$
\Phi\colon \cJ_{C_0(\R)} \to A^{\otimes G}.
$
We can compute the $K$-theory of $\cJ_{C_0(\R)} \rtimes_rG$ as 
\[
K_\ast\left(\cJ_{C_0(\R)} \rtimes_rG\right) \cong K_*(C_r^*(G))\oplus \bigoplus_{[F]\in G\backslash \FINx}K_\ast\left(C_0(\R)^{\otimes F} \rtimes_rG_F\right).\]
In this expression, each nonempty finite subset $F\subseteq G$ can be written as $F=G_F\cdot L_F$ where $L_F$ is a complete set of representatives for $G_F\backslash F$. If the cardinality of $L_F$ is even, the $G_F$-action on $C_0(\R)^{\otimes F}=C_0(\R^{F})$ is $KK^{G_F}$-equivalent to the trivial action on $\bC$ by Kasparov's Bott-periodicity theorem 
 (see \cite[Theorem 20.3.2]{Blackadar}). 
 When the cardinality of $L_F$ is odd, then 
 \[C_0(\R)^{\otimes F}\cong 
 \left(C_0(\R)^{\otimes G_F}\right)^{\otimes |L_F|-1}\otimes C_0(\R)^{\otimes G_F}\]
 is $KK^{G_F}$-equivalent to $C_0(\R)^{\otimes G_F}=C_0\left(\R^{G_F}\right)$. Therefore, we have
\[
K_\ast\left(C_0(\R)^{\otimes F} \rtimes_rG_F\right) \cong  \begin{cases} K_\ast(C^*_r(G_F)), &  \text{if $|G_F\backslash F|$ is even} \\ K_\ast\left(C_0\left(\R^{G_F}\right) \rtimes_rG_F\right), & \text{if $|G_F\backslash F|$ is odd}\end{cases}.
\]
In general, for any finite group $H$ and for any orthogonal representation $H\to O(V)$ on a finite-dimensional Euclidean space $V$, the $K$-theory of $C_0(V)\rtimes_rH$ is the well-studied equivariant topological $K$-theory $K_H^*(V)$ of $V$ (see \cite{Karoubi}, \cite{EP}).
It is known to be a finitely generated free abelian group with $\mathrm{rank}_\Z K_H^\ast(V)$ equal to the number of conjugacy classes $\s{g}$ of $H$ which are oriented and even/odd respectively (see \cite[Theorem 1.8]{Karoubi}). Here, a conjugacy class $\s{g}$ of $H$ is oriented if the centralizer $C_g$ of $g$ acts on the $g$-fixed points $V^{g}$ of $V$ by oriented automorphisms. The class $\s{g}$ is even/odd if the dimension of $V^{g}$ is even/odd respectively. For example for any cyclic group $\Z/m\Z$, using \cite[Example 4.2]{EP}, we have 
\[
K^\ast_{\Z/m\Z}(\R^{m}) = \begin{cases} \Z^{m},  & \ast=1 \\ 0,  & \ast=0  \end{cases},
\]
for odd $m\geq1$, and  
\[
K^\ast_{\Z/m\Z}(\R^{m}) = \begin{cases} \Z^{m/2},  & \ast=1 \\ 0,  & \ast=0  \end{cases} 
\]
for even $m\geq2$. We summarize our discussion as follows:
\end{example}

\begin{theorem} Let $G$ be a discrete group satisfying BCC. We have
\begin{align*}
& K_\ast\left(C(S^1)^{\otimes G}\rtimes_rG\right) \\
 \cong & K_\ast(C^*_r(G)) \oplus \left( \bigoplus_{\underset{|G_F\backslash F| \,\, \mathrm{even}}{[F]\in G\backslash \FIN^\times,}}K_\ast(C^*_r(G_F)) \right) \oplus  \left(\bigoplus_{\underset{|G_F\backslash F| \,\, \mathrm{odd}}{[F]\in G\backslash \FIN^\times,}}K^\ast_{G_F}\left(\R^{G_F}\right) \right).
\end{align*}
\end{theorem}

\begin{example}\label{ex-rotation}  Let $A_\theta$ be the rotation algebra for $\theta \in \R$, the universal $C^*$-algebra generated by two unitaries $u$ and $v$ satisfying $uv=vue^{2\pi i\theta}$. It is well-known that $K_0(A_\theta)\cong \Z^2$, that $K_1(A_\theta)\cong \Z^2$, and that the unital inclusion $\iota\colon \bC \to A_\theta$ induces a split injection on $K_0$. Since $A_\theta$ satisfies the UCT, we can apply Corollary \ref{cor-UCT}
to $B=\bC\oplus C_0(\R) \oplus C_0(\R)$ and obtain
\begin{align*}
 K_\ast &\left(A_\theta^{\otimes G}\rtimes_rG\right)\\
&\quad  \cong K_*(C_r^*(G))\oplus  \bigoplus_{[F]\in G\backslash \FINx}K_\ast\left((\bC\oplus C_0(\R)\oplus C_0(\R))^{\otimes F} \rtimes_rG_F\right).
\end{align*}
 Using Theorem \ref{thm-replace-by-JB-A0} for $B=C_0(\R)\oplus C_0(\R)$, each summand for $F\in \FINx$ may be computed as
\begin{align*}
 &K_\ast\left((\bC\oplus C_0(\R)\oplus C_0(\R))^{\otimes F} \rtimes_rG_F\right) \\
 &\quad\quad\quad\quad\quad\quad\quad\quad\cong \bigoplus_{{[X_1, X_2]}}K_\ast\left(\left(C_0\left(\R^{X_1}\right)\otimes C_0\left(\R^{X_2}\right)\right) \rtimes_rG_{X_{1,2}}\right),
\end{align*}
 where the sum is taken over all ordered equivalence classes $[X_1, X_2]$ of pairs of ordered disjoint subsets $X_1, X_2$ of $F$ modulo the action of  the stabilizer $G_F$ of $F$, and where we set $G_{X_{1,2}}:=G_F\cap G_{X_1}\cap  G_{X_2}$.  As in Example \ref{ex-S1}, the computation of $K_\ast\left(\left(C_0\left(\R^{X_1}\right)\otimes C_0\left(\R^{X_2}\right)  \right)\rtimes_r G_{X_{1,2}}\right)$ either reduces to $K_\ast(C_r^*(G_{X_{1,2}}))$ or to that of $K_\ast\left(C_0\left(\R^{G_{X_{1,2}}}\right)\rtimes_rG_{X_{1,2}}\right)=K^\ast_{G_{X_{1,2}}}\left(\R^{G_{X_{1,2}}}\right)$ in general. We summarize the discussion as follows.
 \end{example}

\begin{theorem} Let $G$ be a discrete group satisfying BCC and let $\theta\in \R$. We have
\[
 K_\ast\left(A_\theta^{\otimes G} \rtimes_rG\right) 
 \cong  K_\ast(C^*_r(G)) \oplus  \bigoplus_{[F]\in G\backslash \FINx} K_F\]
where for all $F\in \FINx$ we set
 \[K_F\coloneqq \left(\bigoplus_{\underset{|G_{X_{1,2}}\backslash (X_1\sqcup X_2)| \,\, \mathrm{even}}{[X_1, X_2], X_1\sqcup X_2\subset F, }} K_\ast(C^*_r(G_{X_{1,2}})) \right) 
 \oplus  \left( \bigoplus_{\underset{ |G_{X_{1,2}}\backslash (X_1\sqcup X_2)| \,\, \mathrm{odd}}{[X_1, X_2], X_1\sqcup X_2\subset F,  } }K^\ast_{G_{X_{1,2}}}\left(\R^{G_{X_{1,2}}}\right)  \right).
 \]
\end{theorem}

\subsection*{Cuntz algebras}
For $n\in \{2,3,\dotsc,\infty\}$ we denote by $\mathcal O_n$ the Cuntz algebra on $n$ generators. 

\begin{example}\cite[Corollary 6.9]{Szabo}\label{ex-Oinfty}
	Let $G$ be a discrete group satisfying BCC. Then the unital inclusion $\bC\to \mathcal O_\infty$ induces an isomorphism
		\[K_*(C^*_r(G))\cong K_*\left(\mathcal O_\infty^{\otimes G}\rtimes_r G\right).\]
\end{example}
\begin{proof}
	Combine Corollary \ref{cor-weak-equivalence} and the fact that the unital inclusion $\bC\to \cO_\infty$ is a $KK$-equivalence. 
\end{proof}

\begin{example}\label{ex-O2}
	Let $G$ be a discrete group satisfying BCC. Then we have 
		\[K_*\left(\mathcal O_2^{\otimes G}\rtimes_r G\right)=0.\]
\end{example}
\begin{proof}
	Combine Lemma \ref{lem-Izumi}, Theorem \ref{thm-goingdown} and the fact that $\cO_2$ is $KK$-equivalent to $0$. 
\end{proof}

\begin{example} \label{ex-On}
Let $G$ be a discrete group satisfying BCC. Let $A=\bC \oplus \cO_n$ for $n\geq3$. By Theorem \ref{thm-replace-by-JB-A0}, we have
\[
K_\ast\left(A^{\otimes G} \rtimes_rG\right)  \cong K_\ast(C^*_r(G)) \oplus \bigoplus_{[F]\in G\backslash \FINx}K_\ast\left(\cO_n^{\otimes F} \rtimes_rG_F\right).
\]
Each summand for $F=G_F\cdot L_F$ becomes $K_\ast\left( \left(\cO_n^{\otimes L_F}\right)^{\otimes G_F} \rtimes_rG_F\right)$. It follows from the UCT and an inductive application of the K\"unneth theorem that $\cO_n^{\otimes L_F}$ is $KK$-equivalent to $\cO_n\otimes (C_0(\R)\oplus \bC)^{|L_F|-1}$. Thus, by Lemma \ref{lem-Izumi}, we  may express $K_\ast(\cO_n^{\otimes F} \rtimes_rG_F)$ explicitly in terms of 
\[
K_\ast(\cO_n^{\otimes H} \rtimes_rH) \,\,\, \text{and} \,\,\, K_\ast( (C_0(\R)\otimes \cO_n)^{\otimes H} \rtimes_rH)
\]
for finite subgroups $H$ of $G_F$. These groups for $H=\Z/2$ are nicely computed in \cite{Izumi}. According to \cite{Izumi}, we have
\[
K_\ast\left(\cO_{n+1}^ {\otimes {\Z/2}} \rtimes_r\Z/2\right)  = \begin{cases} \Z/n \oplus \Z/n,  & \ast=0 \\ 0,  & \ast=1 \end{cases}
\]
for odd $n$, 
\[
K_\ast\left(\cO_{n+1}^ {\otimes	 {\Z/2}} \rtimes_r\Z/2\right)  = \begin{cases} \Z/\frac{n}{2} \oplus \Z/2n,  & \ast=0, \\ 0  & \ast=1  \end{cases}
\]
for even $n$, and 
\[
K_\ast\left(\left(C_0(\R)\otimes \cO_{n+1}\right)^ {\otimes {\Z/2}} \rtimes_r\Z/2\right)   = \begin{cases}0, & \ast=0 \\  \Z/n \oplus \Z/n,   & \ast=1 \end{cases}
\]
for odd $n$, 
\[
K_\ast\left(\left(C_0(\R)\otimes \cO_{n+1}\right)^ {\otimes {\Z/2}} \rtimes_r\Z/2\right)   =\begin{cases} 0, & \ast=0 \\  \Z/\frac{n}{2} \oplus \Z/2n,  & \ast=1  \end{cases}
\]
for even $n$.  Using  similar methods as in \cite{Izumi}, we  can compute the case $H=\Z/3$.  We omit the very technical computations. 
For general $H$, the computations become increasingly complicated as the order of $H$ increases.  

\begin{question} Is $K_\ast\left(\cO_n^{\otimes H} \rtimes_rH\right)$ computable for all finite groups $H$ or at least for all cyclic groups?
\end{question}
\end{example}

Although we do not know how to compute $K_*\left(\cO_n^{\otimes H}\rtimes_r H\right)$ in general, we can say something about its structure. 
The following Corollary is a combination of Theorem \ref{thm-localize} and Corollary \ref{cor-weak-equivalence}.
\begin{corollary}\label{cor-On-localized-is-zero}
	Let $G$ be a discrete group satisfying BCC, $Z$ a countable $G$-set, and $A$ a $C^*$-algebra. Let $M_\mathfrak n$ be a UHF-algebra of infinite type such that $A\otimes M_\mathfrak n$ is $KK$-equivalent to zero (for instance $A=\cO_{n+1}$ and $M_\mathfrak n=M_{n^{\infty}}$ for some $n\geq 2$). Then 
		\[K_*\left(A^{\otimes Z}\rtimes_r G\right)[1/\mathfrak n]\cong 0.\]
\end{corollary}

For $A=\cO_{n+1}$ and a finite group $H$, we can say a bit more:
\begin{proposition}\label{prop-On-fg}
	Let $H$ be a finite group, let $Z$ be a finite $H$-set and let $n\geq 2$. Then $K_*\left(\cO_{n+1}^{\otimes Z}\rtimes_r H\right)$ is a finitely generated abelian group $L$ such that $L[1/n]=0$. That is, any element in $L$ is annihilated by $n^k$ for some $k$. In particular, if $n=p$ is prime, then $L$ is isomorphic to the direct sum of finitely many $p$-groups $\Z/p^k\Z$. 
\end{proposition}

\begin{proof} The method used in \cite{Izumi}, at an abstract level, tells us that \linebreak
$K_\ast\left(\cO_{n+1}^{\otimes Z} \rtimes_rH\right)$ is finitely generated. To see this, let $\mathcal{T}_{n+1}$ be the universal $C^*$-algebra generated by isometries $s_1\cdots, s_{n+1}$ with mutually orthogonal range projections $q_i=s_is_i^\ast$ and we let $p=1-\sum_{1\leq j \leq n+1}q_i$. The ideal generated by $p$ is isomorphic to $\Compacts$ and we have the following exact sequences
 \begin{equation}\label{eq-exact-seq-Cuntz1-H}
 \xymatrix{
0 \ar[r] & I_1\rtimes H \ar[r] & \mathcal{T}_{n+1}^{\otimes Z}\rtimes H   \ar[r] &  \cO_{n+1}^{\otimes Z}\rtimes H \ar[r] &  0,
}
\end{equation}
 \begin{equation}\label{eq-exact-seq-Cuntz2-H}
 \xymatrix{
0 \ar[r] & I_{m+1}\rtimes H \ar[r] &  I_m\rtimes H  \ar[r] &  (I_m/I_{m+1})\rtimes H \ar[r] &  0,
}
\end{equation}
where for $1\leq m\leq |Z|$, an $H$-ideal $I_m$ of $\mathcal{T}_{n+1}^{\otimes Z}$ is defined as the ideal generated by $\Compacts^{\otimes F}\otimes \mathcal{T}_{n+1}^{\otimes  Z-F}$ for subsets $F$ of $H$ with $|F|=m$. In particular $I_{|Z|}= \Compacts^{\otimes Z}$. By Lemma \ref{lem-Izumi}, or by the stabilization theorem, $K_\ast\left(\Compacts^{\otimes Z}\rtimes H\right)\cong K_\ast(C^*(H))$. Moreover, for each $1\leq m<|Z|$, the quotient $I_m/I_{m+1}$ is the direct sum of $\Compacts^{\otimes F}\otimes \cO_{n+1}^{\otimes Z-F}$ over the subsets $F$ of $Z$ with $|F|=m$. Thus, $I_m/I_{m+1}\rtimes H$ is Morita-equivalent to the direct sum of $\Compacts^{\otimes F}\otimes \cO_{n+1}^{\otimes Z-F}\rtimes H_F$ over $[F]\in H\backslash \{ F \subset Z \mid |F|= m \}$ where $H_F$ is the stabilizer of $F$ in $H$. By induction on the size $|Z|$ of $Z$ (for all finite groups $H$ at the same time), $K_\ast\left(\Compacts^{\otimes F}\otimes \cO_{n+1}^{\otimes Z-F}\rtimes H_F\right)\cong K_\ast\left(\cO_{n+1}^{\otimes Z-F}\rtimes H_F\right)$ is finitely generated. Using the six-term exact sequences on $K$-theory associated to \eqref{eq-exact-seq-Cuntz2-H}, we see that $K_\ast(I_m \rtimes H)$ are all finitely-generated. By Lemma \ref{lem-Izumi}, $K_\ast\left( \mathcal{T}_{n+1}^{\otimes Z}\rtimes H \right)\cong K_\ast(C^*(H))$ since $\mathcal{T}_{n+1}$ is $KK$-equivalent to $\bC$ by \cite{Pimsner}.
Using the six-term exact sequence on $K$-theory associated to \eqref{eq-exact-seq-Cuntz1-H}, we now see that  $K_\ast\left( \cO_{n+1}^{\otimes Z}\rtimes H \right)$ is finitely-generated.

On the other hand, we have 
\[
K_\ast\left(\left(\cO_{n+1}^{\otimes Z} \otimes M_{n^\infty}^{\otimes Z}\right) \rtimes_rH\right)  \cong  K_\ast\left( \cO_{n+1}^{\otimes Z} \rtimes_rH\right)[1/n]
\]
by Theorem \ref{thm-localize}. The assertion follows from Lemma \ref{lem-Izumi} since $\cO_{n+1} \otimes M_{n^\infty}$ is $KK$-equivalent to $0$ by the UCT.
\end{proof}

For an infinite $G$-set $Z$, Corollary \ref{cor-On-localized-is-zero} for $A=\cO_{n+1}$ may also be deduced from the combination of Theorem \ref{thm-localize} and the following result (for $A=\cO_{n+1}\otimes M_n$). 

\begin{theorem}\label{ex-0}
Let $G$ be a discrete group satisfying BCC and let $Z$ be a countably infinite $G$-set. Let $A$ be any unital $C^*$-algebra such that $[1_{A^{\otimes r}}]=0\in K_0(A^{\otimes r})$ for some $r\geq 1$. Then we have
\[
K_\ast\left(A^{\otimes Z}\rtimes_rG\right) \cong 0.
\]
\end{theorem}
\begin{proof} By Theorem \ref{thm-goingdown} it is enough to show $K_\ast(A^{\otimes Z}\rtimes_rH) \cong 0$ for all finite subgroups $H$ of $G$. Since $Z$ is infinite, it contains infinitely many orbits of type $H/H_0$ for some fixed subgroup $H_0\subset H$. Denote $Z_0$ the union of orbits of type $H/H_0$. Let $L$ be a (necessarily infinite) complete set of representatives for $H\backslash Z_0$. We have
\[
A^{\otimes Z_0} \cong (A^{\otimes L})^{\otimes H/H_0}.
\] 
By assumption, the unital inclusion $\bC\to A^{\otimes r}$ induces the zero element in $K_0(A^{\otimes r})=KK(\bC,A^{\otimes r})$. In particular, the maps
\[
A^{\otimes N} \to  A^{\otimes {N+r}},
\]
induce the zero map in $KK$-theory since,  on the level  of $KK$-theory, they  are given by
the exterior Kasparov product with $[0]=[1]\in KK(\bC,A^{\otimes r})$. 
It follows from Lemma \ref{lem-Izumi}, that the unital inclusions
\[
\left(A^{\otimes N}\right)^{\otimes H/H_0} \to  \left(A^{\otimes {N+r}}\right)^{\otimes H/H_0}
\]
are zero in $KK^H$. Writing $A^{\otimes Z}=A^{\otimes Z_0}\otimes A^{\otimes Z-Z_0}$, we have
\[
K_\ast\left(A^{\otimes Z}\rtimes_rH\right)  \cong \filcolim_N K_\ast\left(\left(\left(A^{\otimes N}\right)^{\otimes H/H_0} \otimes A^{Z-Z_0} \right)\rtimes_rH\right).
\]
The right-hand side is zero by the preceding argument.
\end{proof}
\begin{remark}\label{rem-ex-0}
Theorem \ref{ex-0} can be applied whenever $A$ is Morita equivalent to $\cO_{n^r+1}$ such that $[1_A]=n\in K_0(A)\cong \Z/{n^r}$, or whenever $K_0(A)\cong \Q/\Z$. In the second case, this is due to the fact that $[1]^{\otimes 2}\in K_0(A^{\otimes 2})$ is in the image of $K_0(A)\otimes_\Z K_0(A)\cong \Q/\Z\otimes_\Z\Q/\Z=0$. 
\end{remark}

\subsection*{Wreath products}
For discrete groups $G$ and $H$, the wreath product $H\wr G$ is defined as the semi-direct product $(\bigoplus_{g\in G} H)\rtimes G$, where $G$ acts by left translation. We have a canonical isomorphism 
	\[C^*_r(H\wr G)\cong C^*_r(H)^{\otimes G}\rtimes_r G.\]
	
An application of Corollary \ref{cor-UCT} gives the following result which generalizes \cite{XinLi}.

\begin{theorem}[Theorem \ref{introthm-wreath}]\label{thm-wreath} Let $G$ be a group satisfying BCC and let $H$ be a group for which $C^*_r(H)$ satisfies the UCT and for which the unital  inclusion $\bC\to C^*_r(H)$ induces a split injection $K_*(\bC)\to K_*(C^*_r(H))$.
Denote its cokernel by $\tilde K_\ast(C^*_r(H))$. Then we have
\[
K_\ast(C^*_r(H\wr G)) \cong K_\ast(C^*_r(G)) \oplus \bigoplus_{[F] \in G\backslash \FINx} K_\ast\left(B^{\otimes F} \rtimes_rG_F\right) 
\]
where $B$ is any $C^*$-algebra satisfying UCT with $K_\ast(B)\cong \tilde K_\ast(C^*_r(H))$. In particular, if $G$ is torsion-free, we have
\begin{align*}
K_\ast(C^*_r(H\wr G)) \cong K_\ast(C^*_r(G))\oplus \bigoplus_{[F]\in G\backslash \FINx}K_\ast\left(B^{\otimes F}\right).
\end{align*}
\end{theorem}

\begin{remark}\label{rem-split-injective}
The assumptions on $H$ in the above theorem are not very restrictive:
If $H$ is a discrete group for which the Baum--Connes assembly map is split-injective (for instance if $H$ satisfies BCC, or  if $H$ is exact by \cite[Theorem 1.1]{Higson-exact}), then the unital inclusion $\bC\to C^*_r(H)$ induces a split injection $K_*(\bC)\to K_*(C^*_r(H))$ since the corresponding map $K^{\topo}_*(\{e\})\to K_*^{\topo}(H)$ always splits. 

On the other hand, it follows from Tu's  \cite[Proposition 10.7]{tu} that $C_r^*(H)$ satisfies the UCT 
for every a-$T$-menable (in particular every amenable) group $H$, or, more generally, if 
$H$ satisfies the strong Baum--Connes conjecture in the sense of \cite{MN}. 
  \end{remark}
  
\begin{example}\label{ex-wreath}
  Let $G$ be a group which satisfies BCC and consider  the wreath product $\mathbb F_n\wr G$ with $\mathbb F_n$ the free group in $n$ generators.  Since $\mathbb F_n$ is known to be a-$T$-menable it follows that Theorem \ref{thm-wreath} applies. Since $K_0(C_r^*(\mathbb F_n))=\Z$ and $K_1(C_r^*(\mathbb F_n))=\Z^n$
we may choose $B=\bigoplus_{i=1}^n C_0(\R)$ so that Theorem \ref{thm-wreath} implies 
\[K_*(C_r^*(\F_n\wr G))\cong K_*(C_r^*(G))\oplus\bigoplus_{F\in G\backslash \FINx} K_*(B^{\otimes F}\rtimes_r G_F),\]
where each summand $K_*(B^{\otimes F}\rtimes_r G_F)$ decomposes into a direct sum of 
equivariant $K$-theory groups of the form $K^*_{H}(V)$ for certain subgroups $H$ of $G_F$ and certain 
euclidean $H$-spaces $V$. A more precise analysis can be done, at least for $n=2$, 
 along the lines of  Example \ref{ex-rotation}.

In particular, if $G$ is torsion free, we get
\[K_*(C_r^*(\F_n\wr G))\cong  K_*(C_r^*(G))\oplus \bigoplus_{[F]\in G\backslash \FINx} K_*(B^{\otimes F})\]
with
$B^{\otimes F} \cong \bigoplus_{i=1}^{n^{|F|}} C_0(\R^{|F|})$. 
Therefore each $G$-orbit of a nonempty finite set $F\subseteq G$ provides $n^{|F|}$ copies  
of $\Z$
as direct summands of  $K_0$ if $|F|$ is even and of  $K_1$ if $|F|$ is odd.
\end{example}

We close this section with

\begin{corollary} Suppose that $G$ and $H$ are as as in Theorem \ref{thm-wreath}, such that the $K$-theory of both $C^*_r(G)$ and $C^*_r(H)$ is free abelian. Then the $K$-theory of $C^*_r(H\wr G)$ is free abelian as well. 
\end{corollary}
\begin{proof} Note that $\tilde K_\ast(C^*_r(H))$ is free abelian as it is the direct summand of the free abelian group $K_\ast(C^*_r(H))$. Therefore, in Theorem \ref{thm-wreath}, $B$ can be taken as the direct sum of, possibly infinitely many, $\bC$ and $C_0(\R)$. The assertion follows from the fact that equivariant $K$-theory $K_{G_0}^\ast(\R^{G_0})$ (more generally $K^*_{G_0}(V)$ for any $G_0$-Euclidean space $V$) is a (finitely-generated) free abelian group for any finite group $G_0$ by \cite{Karoubi} (or \cite{EP}).
\end{proof}

\section{AF-algebras}\label{sec-AF}
Let $A=\filcolim_n A_n$ be a unital AF-algebra (with unital connecting maps) and let $G$ be a discrete group satisfying BCC. If $Z$ is a countable $G$-set, then in principle, we can try to compute 
	\[K_*\left(A^{\otimes Z}\rtimes_r G\right)\cong \filcolim_n K_*\left(A_n^{\otimes Z}\rtimes_r G\right) \]
using the decomposition from Theorem \ref{thm-finitedim-shift}. In general, such a computation can be challenging,  even in relatively simple cases like $A=M_2\oplus M_{3^\infty}$.
Instead of trying to compute the connecting maps, we now provide an 
abstract approach to calculating $K_*(A^{\otimes Z}\rtimes_r G)$ for an arbitrary unital AF-algebra 
$A$, a discrete group $G$ satisfying BCC and a countable proper $G$-set $Z$ 
(for example $Z=G$ with the left translation action). 

We first need some preparation. Let $\mathcal{F}$ be a family of subgroups, i.e. a non-empty set of subgroups of $G$ closed under taking conjugates and subgroups. The orbit category $\OrFG$ has as objects homogeneous $G$-spaces $G/H$ for each $H\in \mathcal{F}$ and as morphisms $G$-maps (see \cite[Definition 1.1]{DL}). We will mainly use the family $\mathcal{FIN}$ of finite subgroups. For any $G$-$C^*$-algebra $A$, we have a functor from $\OrFG$ to the category of graded abelian groups that sends $G/H$ to $K_\ast (A\rtimes_rH)$ and a morphism $G/H_0\to G/{H_1}$ given by $H_0 \mapsto gH_1$ (so that $H_0\subseteq gH_1g^{-1}$) to the map $K_\ast(A\rtimes_rH_0) \to K_\ast(A\rtimes_rH_1)$ induced by the composition 
\[
A\rtimes_rH_0 \hookrightarrow A\rtimes_r{gH_1g^{-1}}\cong A\rtimes_rH_1,\]
where the second map is given by conjugation with $g^{-1}$ inside $A\rtimes_rG$.\footnote{This is a well-defined functor since different choices for $g$ give the same map on $K$-theory. On the other hand, $G/H\mapsto A\rtimes_rH$ does not define a functor. This is why it requires more care to upgrade this to a functor taking values in the category of spectra, see \cite[Section 2]{DL}.}
We denote by
  \[
  \underset{G/H \in \OrFG}{\colim} K_\ast(A\rtimes_rH)
  \] 
  the colimit of the functor $ G/H \mapsto K_\ast(A\rtimes_rH)$ from $\OrFG$ to the category of graded abelian groups.
The inclusions $A\rtimes_rH\to A\rtimes_rG$ induce a natural homomorphism
\begin{equation}\label{eq-0dim-approx}
 \underset{G/H \in \OrFG}{\colim} K_\ast(A\rtimes_rH) \to K_\ast(A\rtimes_rG).
 \end{equation}

\begin{remark}
The map \eqref{eq-0dim-approx} should not be confused with the assembly map
\[
\mathrm{asmb}_{\mathcal{F}}\colon H^G_\ast(E_\mathcal{F} G, \mathbb{K}^{\mathrm{top}}_A) \to K_\ast(A\rtimes_rG),
\]
which corresponds to taking the homotopy colimit 
at the level of $K$-theory spectra instead of taking the ordinary colimit at the level of $K$-theory groups, see \cite[Section 5.1]{DL}.
By the involved universal properties, there is a natural commuting diagram
\[
\xymatrix{
\underset{G/H \in \OrFG}{\colim} K_\ast(A\rtimes_rH) \ar[r] \ar[d] & K_\ast(A\rtimes_rG) \\
 H^G_\ast(E_\mathcal{F} G, \mathbb{K}^{\mathrm{top}}_A)\ar[ru] &
},
\]
but the vertical map is neither injective nor surjective in general. One can think of the map \eqref{eq-0dim-approx} as the best approximation of $K_\ast(A\rtimes_rG)$ using $0$-dimensional $G$-$\mathcal{F}$-CW-complexes. Likewise, the best approximation by $r$-dimensional $G$-$\mathcal{F}$-CW-complexes can be defined by taking the colimit of $H^G_\ast(X, \mathbb{K}^{\mathrm{top}}_A)$ over the category of $r$-dimensional $G$-$\mathcal{F}$-CW-complexes.
\end{remark}

Let us give two examples of how to compute ${\underset{G/H\in \OrFING}\colim K_\ast(A\rtimes_r H)}$ when the structure or $\OrFING$ is understood. Recall that the \emph{coinvariants} of a $G$-module $L$ are given by 
	\[L_G\coloneqq \colim_G L\cong L/\langle x-gx \mid g\in G, x\in L\rangle .\]
Here the group $G$ is considered as a category with one object with morphisms given by the elements of $G$, and $L$ is considered as a functor from $G$ to the category of (graded) abelian groups. \begin{example}
If $G$ is torsion-free, then $\underset{G/H \in \OrFING}\colim K_\ast(A\rtimes_rH)$ can be identified with the coinvariants $K_\ast(A)_G$.
\end{example}

\begin{example}
If $G$ has only one non-trivial conjugacy class $[H]$ of finite subgroups and if $N_G(H)$ 
denotes the normalizer of $H$ in $G$, then  the colimit 
${\underset{G/H \in \OrFING}{\colim}  K_\ast(A\rtimes_rH)}$ is given by the pushout
\[\xymatrix{
	K_*(A)\ar[r]\ar[d]		&K_*(A\rtimes_r H)_{N_G(H)}\ar[d]\\
	K_*(A)_G\ar[r]			&\underset{G/H \in \OrFING}{\colim}  K_\ast(A\rtimes_rH).
}\]
\end{example}

For a $G$-module $M$ and a set $S$ of positive integers, we denote by  
	\[M[S^{-1}]\coloneqq \filcolim (M\xrightarrow{\cdot s_1}M\xrightarrow{\cdot s_2}\cdots)\cong M\otimes_\Z \Z[S^{-1}]\]
its localization at $S$, where $(s_1,s_2,\dotsc)$ is a sequence containing every element of $S$ infinitely many times. Note that localization at $S$ commutes with taking coinvariants since both constructions are colimits. 

\begin{theorem}[Theorem \ref{introthm-AF}]\label{thm-AF-general}
	Let $A$ be a unital $C^*$-algebra of the form $A=\filcolim_n M_{k_n}(A_n)$ (with unital connecting maps) where $(k_n)_{n\in\N}$ is a sequence of positive integers and where each $A_n$ is a unital $C^*$-algebra satisfying the assumptions of Theorem \ref{thm-replace-by-JB-A0} for the unital inclusion $\bC\hookrightarrow A_n$. 
	Let $G$ be an infinite discrete group satisfying BCC, let $Z$ be a countable proper $G$-set, and let $S\subseteq \Z$ be the set of all positive integers $n$ such that $[1_A]\in K_0(A)$ is divisible by $n$. 
	Then, the natural inclusions

\begin{align}
C^*_r(H) \to C^*_r(G),  \,\,\,& C^*_r(H)\to A^{\otimes Z}\rtimes_rH,\nonumber \\
C^*_r(G) \to A^{\otimes Z}\rtimes_rG, \,\,\,&  A^{\otimes Z}\rtimes_rH \to A^{\otimes Z}\rtimes_rG,\label{eq-natural-maps}
\end{align}
induce the following pushout diagram 
\[
\xymatrix{
  \underset{G/H \in \OrFING}{\colim} K_\ast(C^*_r(H))[S^{-1}]    \ar[r] \ar[d] &  \underset{G/H \in \OrFING}{\colim} K_\ast(A^{\otimes Z} \rtimes_rH) \ar[d] \\
 K_\ast(C^*_r(G))[S^{-1}]   \ar[r]  & K_\ast\left(A^{\otimes Z} \rtimes_rG\right).
}
\] 
In particular, if $G$ is torsion-free, this pushout diagram reads
\[
\tilde K_\ast(C^*_r(G))[S^{-1}]  \oplus   K_\ast\left(A^{\otimes Z})_G \cong   K_\ast(A^{\otimes Z} \rtimes_rG\right),
\]
where $\tilde K_\ast(C^*_r(G))$ denotes the cokernel of  $K_\ast(\bC)\to K_\ast(C_r^*(G))$ induced from the unital 
inclusion $\bC\hookrightarrow C_r^*(G)$.
\end{theorem}

\begin{remark}\label{rem-AF-examples}
	By Theorem \ref{thm-coprime-KK}, Theorem \ref{thm-AF-general} applies to all unital AF-algebras. More generally, Theorem \ref{thm-AF-general} applies whenever $A$ is of the form $\filcolim_n M_{k_n}(A_n)$ (with unital connecting maps) where each $A_n$ is one of the following examples:
	\begin{enumerate}
		\item The unitization $B^+$ of a $C^*$-algebra $B$, e.g. $\bC \oplus B$ for unital $B$;
		\item A $C^*$-algebra of the form $\bigoplus_{1\leq j\leq N}C(X_j)\otimes M_{k_j}$ for nonempty
		compact metric spaces $X_j$ and $\gcd(k_1,\dotsc,k_N)=1$ (use Theorem \ref{thm-coprime-KK});
		\item A reduced group $C^*$-algebra $C^*_r(\Gamma)$ of a countable group $\Gamma$ that satisfies the UCT and such that the map $K_*(\bC)\to K_*(C^*_r(\Gamma)))$ is a split-injection (see Remark \ref{rem-split-injective} and Theorem \ref{thm-wreath}).
	\end{enumerate}
	It would be interesting to know if Theorem \ref{thm-AF-general} also holds for unital ASH-algebras, i.e. when $A$ is a filtered colimit of algebras of the form $p(C(X)\otimes M_n) p$ for a projection $p\in C(X)\otimes M_n$. 
\end{remark}

For the proof of Theorem \ref{thm-AF-general}, we need the following Lemma. 

\begin{lemma}\label{lem-colimit-isom} 
Let $G$ be any discrete group, let $\cF$ be a family of subgroups of $G$, let $H_i\in \cF$, and let $A_i$ be $H_i$-$C^*$-algebras for $i\in I$. Denote by $A=\oplus_{i \in I}\mathrm{Ind}_{H_i}^G(A_i)$ the direct sum of the induced $G$-$C^*$-algebras. Then, the natural map
\[
  \underset{G/H \in \OrFG}{\colim}K_\ast(A\rtimes_rH) \to   K_\ast(A\rtimes_rG)
\]
is an isomorphism. 
\end{lemma}
\begin{proof}
By additivity we may assume $A=\mathrm{Ind}_{H_0}^GA_0$ for $H_0\in \cF$ and an $H_0$-$C^*$-algebra $A_0$.
Write 
\[
A=\bigoplus_{[z] \in G/H_0} A_z
\]
where $A_z$ is the fiber of $A$ at $[z] \in G/H_0$. Note that $A_z$ is a $zH_0z^{-1}$-$C^*$-algebra, in particular, $A_e$ is the $H_0$-$C^*$-algebra $A_0$. By \cite[Proposition 2.6.8]{CELY}, the inclusion $A_e\to A$ induces an isomorphism
\[
 K_\ast(A_e\rtimes_rH_0) \cong   K_\ast(A\rtimes_rG).
\]
This isomorphism factors through $\colim_{G/H \in \OrFG}K_\ast(A\rtimes_rH)$ as 
\[
 K_\ast(A_e\rtimes_rH_0) \to K_\ast(A\rtimes_rH_0) \to \underset{G/H \in \OrFG}{\colim} K_\ast(A\rtimes_rH) \to   K_\ast(A\rtimes_rG).
\]
Our claim follows once we show that the map 
\begin{equation}\label{eq-show-surj}
K_\ast(A_e\rtimes_rH_0)  \to \underset{G/H \in \OrFG}{\colim} K_\ast(A\rtimes_rH)
\end{equation}
is surjective. Fix a subgroup $H\in \mathcal F$. Decomposing $G/H_0$ into $H$-orbits and using the 
decomposition of (\ref{eq-decompose-induced}), we obtain a decomposition
\begin{equation}\label{eq-double-cosets}
K_\ast(A\rtimes_rH) \cong \bigoplus_{[z] \in H\backslash G/H_0} K_\ast(A_z\rtimes_r H_z),
\end{equation}
where $H_z\coloneqq H\cap zH_0z^{-1}\in \cF$. 
In \eqref{eq-double-cosets}, the inclusions $K_*(A_z\rtimes_r H_z)\subseteq K_*(A\rtimes_r H)$ are induced by the natural inclusions $A_z\rtimes_r H_z\subseteq A\rtimes_r H$. 
In the colimit $\colim_{G/H \in \OrFG}K_\ast(A\rtimes_rH)$, the summand 
\[
	K_*(A_z\rtimes_r H_z)\subseteq K_*(A\rtimes_r H)
\]
gets via conjugation with $z^{-1}$ identified with the summand 
\[
	K_\ast(A_e\rtimes_rH'_z) \subset K_\ast(A\rtimes_rH'_z)
\]
corresponding to $[e]\in G/H_0$ and $H_z'=z^{-1}Hz\cap H_0$. 
But the elements in $\colim_{G/H \in \OrFG}K_\ast(A\rtimes_rH)$ coming from 
$K_*(A_e\rtimes_r H'_z)$
are certainly in the image of the map in \eqref{eq-show-surj} since
\[K_*(A_e\rtimes_r H'_z)\to \colim_{G/H \in \OrFG}K_\ast(A\rtimes_rH)\]
 factors through the $K$-theory map of the inclusion $A_e\rtimes_rH_z' \subseteq A_e\times_rH_0$.
\end{proof}

As a direct consequence of Lemma \ref{lem-colimit-isom}, we get
\begin{corollary}\label{cor-colimit-isom} Let $G$ be any discrete group, let $Z$ be a countable $G$-set, and let $A_0$ and $B$ be $C^*$-algebras with $A_0$ unital. Let $\dot\cJ_{A_0,B}^{Z}$ be the $G$-$C^*$-algebra 
\[
	\dot\cJ^{Z}_{A_0,B}= \bigoplus_{F\in \FINx(Z)} A_0^{\otimes Z-F}\otimes B^{\otimes F},
	\]
	so that $\cJ_{A_0,B}^Z=A_0^{\otimes Z}\oplus \dot\cJ^Z_{A_0,B}$ as in Definition \ref{defn-JB-A0}.
	Then the natural map
	\[ 
	 \underset{G/H \in \OrFG}{\colim}  K_\ast\left(\dot\cJ^Z_{A_0,B} \rtimes_rH\right) \to K_*\left(\dot\cJ^Z_{A_0,B}\rtimes_rG\right)
	\]
is an isomorphism for any family $\cF$ of subgroups of $G$ containing all the stabilizers of the $G$-action on $\FINx(Z)$. This applies in particular when $Z$ is a proper $G$-set and $\cF = \mathcal{FIN}$ is the family of finite subgroups.
 \qed
	\end{corollary}

\begin{proof}[Proof of Theorem \ref{thm-AF-general}]\label{proof of AF theorem}
	We prove Theorem \ref{thm-AF-general} by considering three successively more general cases: 
	\begin{case}\label{case-1}
		$A$ satisfies the assumptions of Theorem \ref{thm-replace-by-JB-A0} for the unital inclusion $\iota\colon \bC\to A$. 
	\end{case}
	By assumption, there is a $C^*$-algebra $B$ and an element $\phi\in KK(B,A)$ such that 
	$\iota\oplus \phi\in KK(\bC\oplus B,A)$ is a $KK$-equivalence.
	Using the weak $K$-equivalence of $\cJ^Z_B$ and $A^{\otimes Z}$ constructed in Theorem \ref{thm-replace-by-JB-A0}, we may identify the maps in \eqref{eq-natural-maps}
	with the natural maps 
		\begin{equation}\label{eq-natural-maps-JB}
			C^*_r(G)\to \cJ^Z_B\rtimes_r G,\quad \cJ^Z_B\rtimes_r H\to \cJ^Z_B\rtimes_r G,
		\end{equation}
	where the first map is induced from the (non-unital) inclusion $\bC\hookrightarrow \cJ^Z_B$ corresponding to $F=\emptyset$. 
	We may therefore replace $A^{\otimes Z}$ by $\cJ^Z_B$ throughout the proof.
	The corresponding statement for $\cJ^Z_B$ then follows  from Corollary \ref{cor-colimit-isom}
	since by the decomposition $\cJ^Z_B=\bC\oplus\dot\cJ^Z_B$, we have
		\[
	K_\ast( \cJ^Z_B \rtimes_rG) \cong K_\ast(C_r^*(G))\oplus \underset{G/H \in \OrFING}{\colim}  K_\ast(\dot\cJ^Z_B \rtimes_rH)\\
	\]
and consequently a pushout diagram 
	\[
\xymatrix{
  \underset{G/H \in \OrFING}{\colim} K_\ast(C^*_r(H))  \ar[r] \ar[d] &  \underset{G/H \in \OrFING}{\colim}  K_\ast(\cJ^Z_B \rtimes_rH) \ar[d] \\
K_\ast(C^*_r(G))  \ar[r]  & K_\ast( \cJ^Z_B \rtimes_rG).
}
\] 
In the torsion free case this becomes 
	\[
K_*(\cJ^Z_B\rtimes_r G)  \cong K_*(C^*_r(G))\oplus_{K_\ast(\bC)} K_*(\cJ^Z_B)_G\cong \tilde K_*(C^*_r(G)) \oplus K_*(\cJ^Z_B)_G.
	\]
	\begin{case}\label{case-2}
		$A=M_n\otimes D$ where $D$ is as in Case \ref{case-1}.
	\end{case}
	By applying Case \ref{case-1} and localizing at $n$, we see that the maps in \eqref{eq-natural-maps} (for $D$ instead of $A$) induce a pushout diagram 
	\[
\xymatrix{
  \underset{G/H \in \OrFING}{\colim} K_\ast(C^*_r(H))[1/n]  \ar[r] \ar[d] & \underset{G/H \in \OrFING}{\colim}  K_\ast\left(D^{\otimes Z} \rtimes_rH\right)[1/n] \ar[d] \\
 K_\ast(C^*_r(G))[1/n]   \ar[r]  & K_\ast\left(D^{\otimes Z} \rtimes_rG\right)[1/n],
}
\] 
	and in the torsion-free case an isomorphism
		\[\tilde K_*(C^*_r(G))[1/n]\oplus K_*\left(D^{\otimes Z}\right)[1/n]_G\cong K_*\left(D^{\otimes Z}\rtimes_r G\right)[1/n].\]
	Here we have used that localization at $n$ commutes with taking colimits. By Theorem \ref{thm-localize}, the unital inclusion $D\to M_n\otimes D=A$ induces an isomorphism $K_*\left(D^{\otimes Z}\rtimes_r H\right)[1/n]\cong K_*\left(A^{\otimes Z}\rtimes_r H\right)$ for every subgroup $H$ of $G$. This finishes the proof of Case \ref{case-2}.

	\begin{case}\label{case-3}
		$A=\filcolim_k A_k$ where each $A_k$ is as in Case \ref{case-2}.
	\end{case}
	For each $k$, denote by $S_k\subseteq S$ the set of positive integers $n$ such that the unit $[1]\in K_0(A_k)$ is divisible by $n$. Note that we have $S=\bigcup_k S_k$. By Case \ref{case-2}, the conclusion of the theorem holds if we replace $A$ by $A_k$ and $S$ by $S_k$. Now the general case follows by taking the filtered colimit along $k$. 
	\end{proof}

\section{Rational and $k$-adic computations}\label{sec-modulo-k}
In this section we give systematic tools to compute the $K$-theory of noncommutative Bernoulli shifts up to localizing at a supernatural number $\frak n$ (i.e.~at the set of prime factors of $\frak n$).  The results apply to unital $C^*$-algebras $A$ for which the inclusion $\iota \colon \bC\to A$ does not induce a split injection $K_0(\bC)\to K_0(A)$ integrally, but a split injection 
\begin{equation}\label{eq-modulo-split-inj}
K_0(\iota)[1/\frak n]\colon K_0(\bC)[1/\mathfrak n]\hookrightarrow K_0(A)[1/\mathfrak n]
\end{equation}
after localizing at a supernatural number $\mathfrak n$.
Important special cases are $\mathfrak n=k^\infty$ for $k\in \N$ (Example \ref{ex-modulo-k}), or when $\mathfrak n=\prod_p p^{\infty}$ (Example \ref{ex-rational}). The latter case amounts to rational $K$-theory computations since in this case we have $L[1/\mathfrak n]\cong L\otimes_\Z\Q$ for any abelian group $L$.
We give two examples of when one of these situations naturally occurs:
\begin{example}[$M_\mathfrak n=M_{k^{\infty}}$]\label{ex-modulo-k}
	Let $A$ be a unital $C^*$-algebra that admits a finite-dimensional representation $A \to M_k(\bC)$
for some $k$. Then the unital inclusion
$\iota \colon \bC \to A$
induces a split-injection
$K_\ast(\bC)[1/k] \to K_\ast(A)[1/k]$.
Concrete examples are unital continuous trace $C^*$-algebras or subhomogeneous $C^*$-algebras.
\end{example}
\begin{example}[$M_\mathfrak n=\mathcal Q$]\label{ex-rational}
Let $A$ be a unital $C^*$-algebra for which the unit $[1]\in K_0(A)$ is not torsion, for instance let $A$ be unital and stably finite. Then the inclusion $\iota\colon \bC\to A$ induces a split injection $K_*(\bC)\otimes_\Z \Q\to K_*(A)\otimes_\Z \Q$. 
\end{example}

\begin{theorem}[c.f. Theorem \ref{thm-replace-by-JB-A0}]\label{thm-UCT-k}
	Let $G$ be a discrete group satisfying BCC, let $Z$ be a countable $G$-set and let $\mathfrak n$ be a supernatural number. Let $A$ be a unital $C^*$-algebra satisfying the UCT such that the unital inclusion $\bC\to A$ induces a split injection $K_*(\bC)[1/\mathfrak n]\to K_*(A)[1/\mathfrak n]$. Denote by $\tilde K_*(A)$ the cokernel of the injection $K_\ast(\bC) \to K_\ast(A)$ and let $B$ be any $C^*$-algebra satisfying the UCT with $K_*(B)\cong \tilde K_*(A)$. Then we have 
\begin{align*}
	&K_*\left(A^{\otimes Z}\rtimes_r G\right)[1/\mathfrak n]\\
	&\qquad\qquad\cong K_*\left(\cJ^Z_B\rtimes_r G\right)[1/\mathfrak n]\cong \bigoplus_{[F]\in G\backslash \FIN(Z)}K_*\left(B^{\otimes F}\rtimes_r G_F\right)[1/\mathfrak n].
\end{align*}
\end{theorem}
\begin{proof} Replacing $\mathfrak n$ by $\mathfrak n^\infty$, we may assume that $\mathfrak n$ is of infinite type. 
By the UCT, there is a UCT $C^*$-algebra $B$ whose $K$-theory is isomorphic to the cokernel $\tilde K_*(A)$ of $K_*(\iota)$, and an element $\phi\in KK(B\otimes M_\mathfrak n,A\otimes M_\mathfrak n)$ which together with the inclusion $\iota_\mathfrak n\colon M_\mathfrak n\to A\otimes M_\mathfrak n$ induces a $KK$-equivalence
	\begin{equation*}\label{eq-localized-splitting}
		\iota_\mathfrak n\oplus \phi\in KK((\bC\oplus B)\otimes M_\mathfrak n,A\otimes M_\mathfrak n).
	\end{equation*}
We can thus apply Theorem \ref{thm-replace-by-JB-A0} for $A\otimes M_{\mathfrak n}$ in place of $A$, for $A_0=M_{\mathfrak n}$ and for $B\otimes M_{\mathfrak n}$ in place of $B$. We get
\begin{align*}
K_*\left(A^{\otimes Z}\rtimes_r G\right)[1/\mathfrak n] &\cong K_*\left(\left(A\otimes M_{\mathfrak n}\right)^{\otimes Z}\rtimes_r G\right)\\
& \cong K_\ast\left(\cJ_{M_{\mathfrak n}, B\otimes M_{\mathfrak n}}^{Z}\rtimes_r G\right)  \\
& \cong \bigoplus_{[F]\in G\backslash \FIN(Z)}K_*\left( \left(M_{\mathfrak n}^{\otimes Z-F}\otimes(B\otimes M_{\mathfrak n})^{\otimes F}\right)\rtimes_r G_F\right) \\
&\cong  \bigoplus_{[F]\in G\backslash \FIN(Z)}K_*\left(B^{\otimes F}\rtimes_r G_F\right)[1/\mathfrak n]
\end{align*}
where the first and last isomorphisms are obtained from Theorem \ref{thm-localize}. 
\end{proof}

Theorem \ref{thm-AF-general} has a counter-part as well:

\begin{theorem}\label{thm-modulo-k-general}
Let $G$ be a discrete group satisfying BCC, let $Z$ be a countable, proper $G$-set, and let $\mathfrak n$ a supernatural number. Let $A$ be a unital $C^*$-algebra which is a unital filtered colimit of $C^*$-algebras $A_k$ satisfying the assumptions of Theorem \ref{thm-UCT-k}. Then, the natural inclusions
\[
 C^*_r(H) \to C^*_r(G),  \,\,\, C^*_r(H)\to A^{\otimes Z}\rtimes_rH, 
 \]
 \[
C^*_r(G) \to A^{\otimes Z}\rtimes_rG, \,\,\,  A^{\otimes Z}\rtimes_rH \to A^{\otimes Z}\rtimes_rG,
\]
induce a pushout diagram 
\[
\xymatrix{
  \underset{G/H \in \OrFING}{\colim} K_\ast(C^*_r(H))[1/\mathfrak n]    \ar[r] \ar[d] &  \underset{G/H \in \OrFING}{\colim} K_\ast\left(A^{\otimes Z} \rtimes_rH\right)[1/\mathfrak n] \ar[d] \\
 K_\ast(C^*_r(G))[1/\mathfrak n]   \ar[r]  & K_\ast\left(A^{\otimes Z} \rtimes_rG\right)[1/\mathfrak n].
}
\] 
In particular, if $G$ is torsion-free, this pushout diagram reads
\[
\tilde K_\ast(C^*_r(G))[1/\mathfrak n]  \oplus   K_\ast\left(A^{\otimes Z}\right)_G[1/\mathfrak n]  \cong   K_\ast\left(A^{\otimes Z} \rtimes_rG\right)[1/\mathfrak n].
\]
\end{theorem}
\begin{proof} 
Without loss of generality we may assume that $\mathfrak n$ is of infinite type and $A$ satisfies the assumption of Theorem \ref{thm-UCT-k}. 
Let $B,\phi$ and $\iota_\mathfrak n$ be as in the proof of Theorem \ref{thm-UCT-k}.
In particular, Theorem \ref{thm-replace-by-JB-A0} applies so that $(A\otimes M_{\mathfrak n})^{\otimes Z}$ is weakly $K$-equivalent to $\mathcal J_{M_\mathfrak n,B\otimes M_\mathfrak n}^Z$. 
As in the proof of Theorem \ref{thm-AF-general}, the Theorem now follows from the combination of Corollary \ref{cor-colimit-isom} and Theorem \ref{thm-localize}.
\end{proof}

\begin{example} Let $p\in C(X)\otimes M_n$ be a rank-$k$ projection over $C(X)$ corresponding to some vector bundle on $X$. Let $A=p(C(X)\otimes M_n) p$ be the associated homogeneous $C^*$-algebra. Of course, $A$ is Morita-equivalent to $C(X)$ but the unit-inclusion $\iota\colon \bC \to A$ is not split-injective in general; it corresponds to the inclusion $\Z[p] \to K^0(X)$. On the other hand, $A$ has a $k$-dimensional irreducible representation and therefore satisfies  the assumptions of Theorem \ref{thm-UCT-k}. Therefore, we get
\[
K_*\left(A^{\otimes G}\rtimes_r G\right)[1/k]\cong \bigoplus_{[F]\in G\backslash \FIN}K_*\left(B^{\otimes F}\rtimes_r G_F\right)[1/k]
\]
where $B$ is any UCT $C^*$-algebra whose $K$-theory is isomorphic to the cokernel of $K_*(\iota)$.
\end{example}

As an application to Theorem \ref{thm-UCT-k}, we obtain a  proof
 of the fact that Bernoulli shifts by finite groups rarely have the Rokhlin property (see \cite[Definition 3.1]{IzumiRokhlin}). This result is possibly known to experts and we would like to thank N. C. Phillips for mentioning it to us.

\begin{corollary}\label{cor-Rokhlin}
	Let $A$ be a unital $C^*$-algebra satisfying the UCT such that $[1]\in K_0(A)$ is not torsion (for instance, suppose that $A$ is stably finite). Let $G\neq \{e\}$ be a finite group and let $Z$ a $G$-set. Then the Bernoulli shift of $G$ on $A^{\otimes Z}$ does not have the Rokhlin property. 
\end{corollary}
\begin{proof}
	It follows from Theorem \ref{thm-UCT-k} that the inclusion $C^*_r(G)\hookrightarrow A^{\otimes Z}\rtimes_r G$ induces a split injection 
		\begin{equation}\label{eq-rationally-injective}
			K_*(C^*_r(G))\otimes_\Z\Q\hookrightarrow K_*\left(A^{\otimes Z}\rtimes_r G\right)\otimes_\Z\Q.
		\end{equation}
	Now assume that the action of $G$ on $A^{\otimes Z}$ has the Rokhlin property. Using \cite[Theorem 2.6]{Phillips}, we can find a unital equivariant $*$-homomorphism $C(G)\to A^{\otimes Z}$.
	But then we can factor the map in \eqref{eq-rationally-injective} as the composition
	\[K_*(C^*_r(G))\otimes_\Z \Q\to K_*(C(G)\rtimes_r G)\otimes_\Z\Q\to K_*\left(A^{\otimes Z}\rtimes_r G\right)\otimes_\Z\Q,\]
	which is never injective unless $G=\{e\}$.
\end{proof}

%%%%%%%%%%%%%%%%%%%%%%%%%%%%%%%%%%%%%%%%%%%%%%%%

%\bibliography{Refs}
\bibliographystyle{alpha}
\def\cprime{$'$}

\end{document}